\documentclass[10pt]{article}
\usepackage[latin1]{inputenc}
\usepackage[T1]{fontenc}

\usepackage[bitstream-charter]{mathdesign}

\usepackage{hyperref}
\hypersetup{
    hyperfootnotes=false,
    colorlinks=true,
    linkcolor=black,
    citecolor=black,
    filecolor=black,
    urlcolor=black
}

\usepackage{amsmath}
\usepackage{amsthm}
\swapnumbers
\usepackage[alphabetic]{amsrefs}

\usepackage{bm}
\usepackage{extarrows}
\usepackage[all]{xy}
\usepackage{tikz}
\usepackage{mathtools}
\usepackage{enumitem}

\usepackage[title,titletoc]{appendix}

\usepackage{fixltx2e}

\usepackage{geometry}

\usepackage{microtype}

\theoremstyle{plain}
\newtheorem{thm}{Theorem}[section]
\newtheorem{lma}[thm]{Lemma}
\newtheorem{cor}[thm]{Corollary}
\newtheorem{prop}[thm]{Proposition}
\newtheorem*{announceconj}{Conjecture}

\theoremstyle{definition}
\newtheorem{dfn}[thm]{Definition}
\newtheorem{cnstr}[thm]{Construction}

\theoremstyle{remark}
\newtheorem{rem}[thm]{Remark}
\newtheorem{ex}[thm]{Example}
\newtheorem{conv}[thm]{Convention}
\newtheorem{notn}[thm]{Notation}

\DeclareMathOperator{\Chow}{CH}
\DeclareMathOperator{\codim}{codim}

\DeclareMathOperator{\Cyc}{Z}
\DeclareMathOperator{\Db}{D^b}
\DeclareMathOperator{\Dperf}{D^{perf}}
\DeclareMathOperator{\Hom}{Hom}
\DeclareMathOperator{\id}{id}
\DeclareMathOperator{\im}{im}

\DeclareMathOperator{\Kzero}{K_0}
\DeclareMathOperator{\Min}{Min}
\DeclareMathOperator{\Mod}{\!-mod}
\DeclareMathOperator{\proj}{\!-proj}
\DeclareMathOperator{\Proj}{Proj}
\DeclareMathOperator{\Spc}{Spc}
\DeclareMathOperator{\stab}{\!-stab}
\DeclareMathOperator{\supp}{supp}

\begin{document}
	
\title{Intersection products for tensor triangular Chow groups}
\author{Sebastian Klein\thanks{The author acknowledges the support of the European Union for the ERC grant No.\ 257004-HHNcdMir.}}
\date{}

\maketitle

\begin{abstract}
We show that under favorable circumstances, one can construct an intersection product on the Chow groups of a tensor triangulated category $\mathcal{T}$ (as defined in \cite{balmerchow}) which generalizes the usual intersection product on a non-singular algebraic variety. Our construction depends on the choice of an algebraic model for $\mathcal{T}$ (a \emph{tensor Frobenius pair}), which has to satisfy a $\mathrm{K}$-theoretic regularity condition analogous to the Gersten conjecture from algebraic geometry. In this situation, we are able to prove an analogue of the Bloch formula and use it to define an intersection product similar to Grayson's construction from \cite{graysonproduct}. We then recover the usual intersection product on a non-singular algebraic variety assuming a $\mathrm{K}$-theoretic compatibility condition.
\end{abstract}

\tableofcontents

\pagestyle{headings}

\section{Introduction}

In \cite{balmerchow}, the Chow groups $\Chow^{\Delta}_p(\mathcal{T})$ of a tensor triangulated category $\mathcal{T}$ were introduced and it was shown in \cite{kleinchow} that they have a lot of desirable properties, in analogy with the situation in algebraic geometry. The intersection product, one of the most important operations on the Chow groups of a non-singular algebraic variety, however, did not have an analogue in the tensor triangular world yet. In this article, we give a construction that provides us --- under favorable circumstances --- with an intersection product for a tensor triangulated category $\mathcal{T}$, that is defined on groups $\prescript{}{\cap}\Chow^{\Delta}_p(\mathcal{T}) \subset \Chow^{\Delta}_p(\mathcal{T})$ (see Definition \ref{dfnaltchow2}). In the case that $\mathcal{T} = \Dperf(X)$ for a separated, non-singular scheme $X$ of finite type over a field, the groups $\prescript{}{\cap}\Chow^{\Delta}_p(\mathcal{T})$ coincide with $\Chow^{\Delta}_p(\mathcal{T})$ (see Lemma \ref{lmachoweq}) and thus recover the usual Chow groups $\Chow^p(X)$ by \cite{kleinchow}*{Theorem 3.2.6}. 

In order to define the intersection product, the category $\mathcal{T}$ should satisfy two conditions: Firstly, $\mathcal{T}$ should have an ``algebraic model'' in the sense that there should exist a tensor Frobenius pair $\mathbf{A}$ (see Definition \ref{tensorfrobdef}) with derived category $\mathcal{T}$. Following Schlichting \cite{schlichtingnegative}, the assumption that $\mathcal{T}$ has a Frobenius pair as a model gives us the tools of the higher and negative algebraic $\mathrm{K}$-theory of the model. They allow us to define $\mathrm{K}$-theory sheaves $\mathscr{K}^{(0)}_p$ on the spectrum $\Spc(\mathcal{T})$ (see Definition \ref{ksheafdef}). Our second (and more restrictive) assumption concerns the behavior of a localization sequence arising from the $\mathrm{K}$-theory of the Frobenius models associated to certain sub-quotients of $\mathcal{T}$, and states that an analogue of the Gersten conjecture from algebraic geometry should hold (see Definition~\ref{gersten}). Under this assumption, we can construct a partial flasque resolution of the sheaf $\mathscr{K}^{(0)}_p$ and calculate
\[\mathrm{H}^{p}(\Spc(\mathcal{T}),\mathscr{K}^{(0)}_p) \cong \prescript{}{\cap}\Chow^{\Delta}_p(\mathcal{T})\]
(see Theorem \ref{bloch}), a result analogous to the usual Bloch formula from algebraic geometry. We combine this identification with the cup product from sheaf cohomology and a map
\[\mathscr{K}^{(0)}_p \otimes \mathscr{K}^{(0)}_q \to \mathscr{K}^{(0)}_{p+q} \]
induced by the product in Waldhausen $\mathrm{K}$-theory to obtain an intersection product
\[\prescript{}{\cap}\Chow^{\Delta}_p(\mathcal{T}) \otimes \prescript{}{\cap}\Chow^{\Delta}_q(\mathcal{T}) \to \prescript{}{\cap}\Chow^{\Delta}_{p+q}(\mathcal{T}) \]
(see Definition \ref{productdef}). Using the identification $\prescript{}{\cap}\Chow^{\Delta}_p(\Dperf(X)) = \Chow^p(X)$ for a separated, non-singular scheme $X$ of finite type over a field, we show that the product coincides with the usual intersection product on $X$ (see \ref{thmprodagrees}), by recurring to a similar theorem of Grayson (see \cite{graysonproduct}). For the proof we need to assume a compatibility condition between the products in the Waldhausen and Quillen $\mathrm{K}$-theory of a certain exact category.

The appendix discusses the countable envelope of a tensor Frobenius pair, a construction that extends work of Keller \cite{kellerchain}*{Appendix B} and Schlichting \cite{schlichtingnegative}*{Section 4} and that we need in order to obtain an algebraic model for the idempotent completion of a tensor triangulated category (see Section \ref{sectmodic}).

\textbf{Acknowledgements:} The results in this paper are all taken from the author's Ph.D.\ thesis, which was written at Utrecht University and jointly supervised by Paul Balmer and Gunther Cornelissen. The author would like to thank both of them for their support. The members of the thesis committee also pointed out some corrections and suggested some improvements, for which the author is grateful.

\section{Preliminaries}
\subsection{Tensor triangular geometry}
Let us recall some basics of tensor triangulated categories and \emph{tensor triangular geometry}, our main tool to study them.
\begin{dfn}[see \cite{balmer2010tensor}*{Definition 3}] 
	A tensor triangulated category is an essentially small triangulated category $\mathcal{T}$ endowed with a compatible symmetric monoidal structure. That is, there is a bifunctor
	\[\otimes: \mathcal{T} \times \mathcal{T} \to \mathcal{T}\]
	and a unit object $\mathbb{I}$, together with associator, unitor and commutator isomorphisms: 
	for all objects $X,Y,Z$ in $\mathcal{T}$, we have natural isomorphisms 
	\[X \otimes (Y \otimes Z) \cong (X \otimes Y) \otimes Z, \quad X \otimes \mathbb{I} \cong X \cong \mathbb{I} \otimes X, \quad X \otimes Y \cong Y \otimes X\]
	that satisfy the coherence conditions of \cite{maclanecwm}*{Section XI.1} to make $\mathcal{T}$ a symmetric monoidal category. Furthermore, the bifunctor $\otimes$ is required to be exact in each variable.
	\label{defttcat}
\end{dfn}

\begin{dfn}
	A triangulated category $\mathcal{T}$ is called \emph{idempotent complete} if all idempotent endomorphisms in $\mathcal{T}$ split. A full subcategory $\mathcal{J} \subset \mathcal{T}$ is called \emph{dense} if every object of $\mathcal{T}$ is a direct summand of an object of $\mathcal{J}$.
\end{dfn}
Recall that we can always embed a (tensor) triangulated category $\mathcal{T}$ into its \emph{idempotent completion} $\mathcal{T}^{\natural}$ as a dense subcategory. The category $\mathcal{T}^{\natural}$ is idempotent complete and naturally (tensor) triangulated, and the essential image of $\mathcal{T}$ in $\mathcal{T}^{\natural}$ is dense (see \cite{balschlichidem} and \cite{balmer2005spectrum}*{Remark 3.12}). Idempotent completions are functorial and if $\mathcal{J} \hookrightarrow \mathcal{T}$ is the inclusion of a full dense subcategory, then the induced functor $\mathcal{J}^{\natural} \hookrightarrow \mathcal{T}^{\natural}$ is an equivalence. 

The starting point for studying tensor triangulated categories geometrically is the \emph{spectrum} of a tensor triangulated category, whose construction we briefly describe next.
\begin{dfn} 
	Let $\mathcal{T}$ be a tensor triangulated category. A thick triangulated subcategory $\mathcal{J} \subset \mathcal{T}$ is called
	\begin{itemize}
		\item \emph{$\otimes$-ideal} if $\mathcal{T} \otimes \mathcal{J} \subset \mathcal{J}$.
		\item \emph{prime} if $\mathcal{J}$ is a proper $\otimes$-ideal ($\mathcal{J} \neq \mathcal{T}$) and $A \otimes B \in \mathcal{J}$ implies $A \in \mathcal{J}$ or $B \in \mathcal{J}$ for all objects $A,B \in \mathcal{T}$.
	\end{itemize}
\end{dfn}

\begin{dfn}[see \cite{balmer2005spectrum}] 
	Let $\mathcal{T}$ be an essentially small tensor triangulated category. The \emph{spectrum} of $\mathcal{T}$ is the set
	\[\Spc(\mathcal{T}) := \lbrace \mathcal{P} \subset \mathcal{T}: \mathcal{P}~\text{is a prime ideal} \rbrace\]
	topologized by the basis of closed sets of the form
	\[\supp(A) := \lbrace \mathcal{P} \in \Spc(\mathcal{T}): A \notin \mathcal{P} \rbrace\]
	for objects $A \in \mathcal{T}$. The set $\supp(A)$ is called the \emph{support of $A$}.
\end{dfn}

The spectrum $\Spc(\mathcal{T})$ is well-behaved: it is always a spectral topological space and it behaves (contravariantly) functorially with respect to the class of $\otimes$-exact functors (see e.g.\ \cite{balmer2010tensor}). Let us give some computations of $\Spc(\mathcal{T})$ for the purpose of illustration.

\begin{ex}
	The idempotent completion $\mathcal{T}^{\natural}$ is naturally a tensor triangulated category and the map $\Spc(\mathcal{T}^{\natural}) \to \Spc(\mathcal{T})$ induced by the inclusion $\mathcal{T} \hookrightarrow \mathcal{T}^{\natural}$ is a homeomorphism (see \cite{balmer2005spectrum}*{Corollary~3.14}).
	\label{exidcomptt}
\end{ex}

\begin{ex}[see \cite{balmer2005spectrum}, \cite{bkssupport}*{Theorem 9.5}]
	Let $X$ be a quasi-compact, quasi-sep\-a\-rated scheme and let $\Dperf(X)$ denote the derived category of perfect complexes on $X$. This is a tensor triangulated category with tensor product $\otimes^{\mathrm{L}}$. We have $\Spc(\Dperf(X)) \cong X$ and moreover, the support $\supp(A^{\bullet})$ of a complex $A^{\bullet} \in \Dperf(X)$ coincides with the support of the total homology sheaf $\bigoplus_i \mathrm{H}^i(A^{\bullet})$ on $X$ under this isomorphism. The proof of the statement uses Thomason's classification result from \cite{thomasonclassification}.
	\label{exschemereconstruct}
\end{ex}

\begin{ex}[see \cite{balmer2005spectrum}*{Corollary 5.10}]
	Let $G$ be a finite group and $k$ be a field such that $\mathrm{char}(k)$ divides the order of $G$. Let $kG\stab$ denote the stable module category, i.e.\ the category of finitely generated $kG$-left modules with $\Hom_{kG\stab}(M,N) = \Hom_{kG\Mod}(M,N)/\mathcal{J}$, where $\mathcal{J}$ denotes the subgroup of morphisms that factor through a projective module. This is a tensor triangulated category with tensor product $\otimes_k$. We have $\Spc(kG\stab) \cong \mathcal{V}_G(k)$, the \emph{projective support variety} of $k$. The variety $\mathcal{V}_G(k)$ is defined as $\Proj(\mathrm{H}^*(G,k))$, where $\mathrm{H}^*(G,k)$ denotes the cohomology ring of $G$ over $k$. The support $\supp(M)$ of a module $M \in kG\stab$ coincides with the cohomological support of $M$ in $\mathcal{V}_G(k)$ under this isomorphism. The proof of the statement uses the classification of thick $\otimes$-ideals in $kG\stab$ from \cite{bencarlrick}.
	\label{exstablemodcat}
\end{ex}

The spectrum of a tensor triangulated category $\mathcal{T}$ gives us the possibility to assign to each object of $\mathcal{T}$ the dimension of its support.
\begin{dfn}[see \cite{balmerfiltrations}] 
	A \emph{dimension function} on $\mathcal{T}$ is a map 
	\[\dim: \Spc(\mathcal{T}) \to \mathbb{Z} \cup \lbrace \pm \infty \rbrace\]
	such that the following two conditions hold:
	\begin{enumerate}
		\item If $\mathcal{Q} \subset \mathcal{P}$ are prime tensor ideals of $\mathcal{T}$, then~$\dim(\mathcal{Q}) \leq \dim(\mathcal{P})$.
		\item If $\mathcal{Q} \subset \mathcal{P}$ and $\dim(\mathcal{Q}) = \dim(\mathcal{P}) \in \mathbb{Z}$, then~$\mathcal{Q} = \mathcal{P}$.
	\end{enumerate}
	For a subset $V \subset \Spc(\mathcal{T})$, we define $\dim(V) := \sup \lbrace \dim(\mathcal{P}) | \mathcal{P} \in V \rbrace$. For every $p \in \mathbb{Z} \cup \lbrace \pm \infty \rbrace$, we define the full subcategory
	\[\mathcal{T}_{(p)} := \lbrace a \in \mathcal{T} : \dim(\supp(a)) \leq p \rbrace ~.\]
	We denote by $\Spc(\mathcal{T})_p$ the set of points $\mathcal{Q}$ of $\Spc(\mathcal{T})$ such that $\dim(\mathcal{Q}) = p$.
	\label{dimfuncdef}
\end{dfn}

\begin{rem}
	From the properties of $\supp(-)$, it follows that $\mathcal{T}_{(p)}$ is a thick tensor ideal in~$\mathcal{T}$. 
\end{rem}

\begin{ex}
	The main examples of dimension functions we will consider are the Krull dimension and the opposite of the Krull co-dimension. For $\mathcal{P} \in \Spc(\mathcal{T})$, its \emph{Krull dimension} $\dim_{\mathrm{Krull}}(\mathcal{P})$ is the maximal length $n$ of a chain of irreducible closed subsets
	\[\emptyset \subsetneq C_0 \subsetneq C_1 \subsetneq \ldots \subsetneq C_n = \overline{\lbrace \mathcal{P} \rbrace}. \]
	Dually, we define the \emph{opposite of the Krull co-dimension} 
	\[-\codim_{\mathrm{Krull}}(\mathcal{P})\] 
	as follows: if we have a chain of irreducible closed subsets of maximal length
	\[ \overline{\lbrace \mathcal{P} \rbrace} = C_0 \subsetneq C_1 \ldots \subsetneq C_n = \text{ maximal irred.~comp.~of $\Spc(\mathcal{T})$ containing $\mathcal{P}$}\]
	we set 
	\[-\codim_{\mathrm{Krull}}(\mathcal{P}) = -n~.\]
	\label{dimfuncex}
\end{ex}

A dimension function determines a filtration of $\mathcal{T}$. We have a chain of $\otimes$-ideals 
\[\mathcal{T}_{(-\infty)} \subset \cdots \subset \mathcal{T}_{(p)} \subset \mathcal{T}_{(p+1)} \subset \cdots \subset \mathcal{T}_{(\infty)} = \mathcal{T}~.\]
The sub-quotients of this filtration have a local description which we will recall next. First, let us introduce another useful property of tensor triangulated categories.
\begin{dfn}[see \cite{balmer2010tensor}*{Definition 20}]
	A tensor triangulated category $\mathcal{T}$ is called \emph{rigid} if there is an exact functor $D: \mathcal{T}^{\mathrm{op}} \to \mathcal{T}$ and a natural isomorphism $\Hom_{\mathcal{T}}(a \otimes b, c) \cong \Hom_{\mathcal{T}}(b, D(a) \otimes c)$ for all objects $a,b,c \in \mathcal{T}$. The object $D(a)$ is called the \emph{dual} of $a$.
	\label{dfnrigidttcat}
\end{dfn}

\begin{ex}
	Both categories $\Dperf(X)$ and $kG\stab$ (see Examples \ref{exschemereconstruct} and \ref{exstablemodcat}) are rigid. In these cases, the functor $D$ is given by $\mathrm{R}\mathcal{H}\!om(-,\mathcal{O}_X)$ and $\Hom_{k}(-,k)$ respectively.
\end{ex}

\begin{rem}
	From the natural isomorphism
	\[\Hom_{\mathcal{T}}(a \otimes b, c) \cong \Hom_{\mathcal{T}}(b, D(a) \otimes c)\]
	of Definition \ref{dfnrigidttcat}, it follows that  $a \otimes -$ and $D(a) \otimes -$ form an adjoint pair of functors for all objects $a \in \mathcal{T}$.
	\label{remrigidadjoint}
\end{rem}

We now fix a dimension function on a rigid tensor triangulated category $\mathcal{T}$ and look at the sub-quotients of the induced filtration. They have a local description.
\begin{thm}[see \cite{balmerfiltrations}*{Theorem 3.24}] 
	Let $\mathcal{T}$ be a rigid tensor triangulated category equipped with a dimension function $\dim$ such that $\Spc(\mathcal{T})$ is a noetherian topological space. Then, for all $p \in \mathbb{Z}$, there is an exact equivalence
	\[\left(\mathcal{T}_{(p)}/\mathcal{T}_{(p-1)}\right)^{\natural} \to \coprod_{\substack{P \in \Spc(\mathcal{T}) \\ \dim(P) = p}} \Min(\mathcal{T}_P)~.\]
	where $\mathcal{T}_P:= (\mathcal{T}/P)^{\natural}$ and $\Min(\mathcal{T}_P)$ denotes the full triangulated subcategory of objects with support the unique closed point of $\Spc(\mathcal{T}_P)$.
	\label{thmfiltdecomp}
\end{thm}

\begin{rem}
	The exact equivalence of Theorem \ref{thmfiltdecomp} is induced by the functor.
	\begin{align*}
	\mathcal{T}_{(p)}/\mathcal{T}_{(p-1)} &\to \coprod_{\substack{\mathcal{P} \in \Spc(\mathcal{T}) \\ \dim(\mathcal{P}) = p}} \Min(\mathcal{T}/\mathcal{P})\\
	a &\mapsto (Q_{\mathcal{P}}(a)) 
	\end{align*}
	where $Q_{\mathcal{P}}$ is the localization functor $\mathcal{T} \to \mathcal{T}/\mathcal{P}$. It is shown in \cite{balmerfiltrations} that the image of this functor is dense, so it induces an equivalence after idempotent completion on both sides.
\end{rem}

\begin{cnstr}[Dimension functions and localizations]
	Assume we are given a rigid tensor triangulated category $\mathcal{T}$ equipped with a dimension function $\dim$ and a thick $\otimes$-ideal $I \subset \mathcal{T}$. By the classification of thick $\otimes$-ideals in $\mathcal{T}$ (see \cite{balmer2005spectrum}*{Theorem 4.10}), there exists a Thomason subset $Z \subset \Spc(\mathcal{T})$ (i.e.\ $Z$ can be written as a union of closed subsets with quasi-compact complement) such that
	\[I=\mathcal{T}_Z = \lbrace a \in \mathcal{T} | \supp(a) \subset Z\rbrace ~.\]
	Let $q: \mathcal{T} \to \mathcal{T}/\mathcal{T}_Z$ be the Verdier localization functor and denote by $U$ the complement of $Z$ in $\Spc(\mathcal{T})$. By \cite{balmer2005spectrum}*{Proposition 3.11}, $q$ induces a homeomorphism 
	\[\Spc(q): \Spc(\mathcal{T}/\mathcal{T}_Z) \xlongrightarrow{\sim} U \subset \Spc(\mathcal{T})\]
	which we use to define an induced dimension function $\dim|_U$ on $\mathcal{T}/\mathcal{T}_Z$ by setting
	\[\dim|_U(\mathcal{P}) := \dim(\Spc(q)(\mathcal{P}))\]
	for $\mathcal{P} \in \Spc(\mathcal{T}/\mathcal{T}_Z)$.
	\label{constrresdim}
\end{cnstr}

The dimension function $\dim|_U$ enables us to also filter the quotient $\mathcal{T}/\mathcal{T}_Z$ by dimension of support. Let us prove that the quotient functor $q$ is compatible with it.
\begin{lma}
	Let  $\mathcal{T}$ be a rigid tensor triangulated category equipped with a dimension function $\dim$ and let $\mathcal{T}_Z\subset \mathcal{T}$ be a thick tensor ideal as in Construction \ref{constrresdim}. Equip $\mathcal{T}/\mathcal{T}_Z$ with the dimension function $\dim_U$ and let $q: \mathcal{T} \to \mathcal{T}/\mathcal{T}_Z$ denote the localization functor. Then $q(\mathcal{T}_{(p)}) \subset (\mathcal{T}/\mathcal{T}_Z)_{(p)}$ for all $p \in \mathbb{Z}$. (In the terminology of \cite{kleinchow}, the functor $q$ \emph{has relative dimension 0}).
	\label{lmareshasreldim0}
\end{lma}
\begin{proof}
	Let $a$ be an object in $\mathcal{T}_{(p)}$. Then we have
	\[\dim|_U(\supp(q(a))) = \dim|_U(\supp(a) \cap U) = \dim(\supp(a) \cap U) \leq \dim(\supp(a)) \leq p\]
	where the first equality follows from \cite{balmer2005spectrum}*{Proposition 3.6} as $\Spc(q)$ is the inclusion~$U \hookrightarrow \Spc(\mathcal{T})$.
\end{proof}

The dimension function$-\codim_{\mathrm{Krull}}$ is well-behaved with respect to restriction:
\begin{lma}
	Let $\mathcal{T}$ be a rigid tensor triangulated category equipped with the dimension function $-\codim^{\mathcal{T}}_{\mathrm{Krull}}$ and $\mathcal{T}_Z\subset \mathcal{T}$ be a thick tensor ideal as in Construction \ref{constrresdim}. Then the dimension functions $-\codim^{\mathcal{T}}_{\mathrm{Krull}}|_U$ and $-\codim^{\mathcal{T}/\mathcal{T}_Z}_{\mathrm{Krull}}$ on $\mathcal{T}/\mathcal{T}_Z$ coincide.
	\label{lmarescodimiscodim}
\end{lma}
\begin{proof}
	The Verdier quotient functor $q: \mathcal{T} \to \mathcal{T}/\mathcal{T}_Z$ induces an inclusion-preserving bijection between the prime ideals of $\mathcal{T}/\mathcal{T}_Z$ and the prime ideals of $\mathcal{T}$ containing $\mathcal{T}_Z$ (see \cite{balmer2005spectrum}*{Proposition 3.11}). From this, the claim follows immediately, as $-\codim_{\mathrm{Krull}}|$ counts prime ideals in a chain that contain a given one.
\end{proof}

\begin{rem}
	Lemma \ref{lmarescodimiscodim} is not true if we replace $-\codim_{\mathrm{Krull}}$ by $\dim_{\mathrm{Krull}}$. For example, let $\mathcal{P} \in \Spc(\mathcal{T})$ be a prime ideal with $\dim_{\mathrm{Krull}}(\mathcal{P})=c>0$ and set $\mathcal{T}_Z := \mathcal{P}$. Then
	\[\dim_{\mathrm{Krull}}|_U(0)= \dim_{\mathrm{Krull}}(\mathcal{P}) = c\]
	but the Krull dimension of the prime ideal $0 \in \Spc(\mathcal{T}/\mathcal{T}_Z)$ is~0.
\end{rem}

\subsection{Tensor triangular Chow groups}
In \cite{balmerchow}, the following notion of Chow groups for a tensor triangulated category was introduced. Let $\mathcal{T}$ be a tensor triangulated category as in Definition \ref{defttcat}, equipped with a dimension function and let $p \in \mathbb{Z}$. Consider the diagram
\[\xymatrix{
	\mathcal{T}_{(p)} \ar@{^{(}->}[r]^I \ar@{->>}[d]^Q& \mathcal{T}_{(p+1)} \\
	\mathcal{T}_{(p)}/\mathcal{T}_{(p-1)} \ar@{^{(}->}[r]^J & (\mathcal{T}_{(p)}/\mathcal{T}_{(p-1)})^{\natural}
}
\]
where $I,J$ denote the obvious embeddings and $Q$ is the Verdier quotient functor. After applying $\Kzero$ we get a diagram
\[
\xymatrix{
	\Kzero(\mathcal{T}_{(p)}) \ar[r]^i \ar@{->>}[d]^q& \Kzero(\mathcal{T}_{(p+1)}) \\
	\Kzero(\mathcal{T}_{(p)}/\mathcal{T}_{(p-1)}) \ar@{^{(}->}[r]^(0.38){j} & \Kzero\left( (\mathcal{T}_{(p)}/\mathcal{T}_{(p-1)})^{\natural} \right) 
}
\]
where the lowercase maps are induced by the uppercase functors.

\begin{dfn}[see \cite{balmerchow}*{Definitions 8 and 10}]
	The \emph{$p$-dimensional cycle group of $\mathcal{T}$} is defined as
	\[\Cyc^{\Delta}_{p}(\mathcal{T}) := \Kzero\left( (\mathcal{T}_{(p)}/\mathcal{T}_{(p-1)} )^{\natural}\right)~.\]
	The \emph{$p$-dimensional Chow group of $\mathcal{T}$} is defined as
	\[\Chow^{\Delta}_{p}(\mathcal{T}) := \Cyc^{\Delta}_{p}(\mathcal{T}) / j \circ q (\ker(i))~.\]
	\label{defcycle} \label{defchow}
\end{dfn}

Let us recall from \cite{kleinchow} the following theorem that justifies the name "Chow groups":
\begin{thm}[see \cite{kleinchow}*{Theorem 3.2.6}]
	Let $X$ be a separated, non-singular scheme of finite type over a field and assume that the tensor triangulated category $\Dperf(X)$ is equipped with the dimension function~$-\codim_{\mathrm{Krull}}$. Then there are isomorphisms
	\[\Cyc^{\Delta}_{p}\left(\Dperf(X)\right) \cong Z^{-p}(X) \qquad \text{and} \qquad \Chow^{\Delta}_{p}\left(\Dperf(X)\right) \cong \Chow^{-p}(X) \]
	for all $p \in \mathbb{Z}$.
	\label{agreementthm}
\end{thm}

\subsection{Algebraic models for tensor triangulated categories}
\label{sectionmodels}
From now on, let $\mathcal{T}$ denote an essentially small tensor triangulated category as in Definition~\ref{defttcat}. It is well-known that there is no $\mathrm{K}$-theory functor from the category of small triangulated categories to the category of spaces, if we require that it satisfies some natural axioms (see \cite{schlichtingknote}). In order to be able to talk about the higher and negative $\mathrm{K}$-theory of $\mathcal{T}$, we therefore work with an algebraic model of $\mathcal{T}$, rather than $\mathcal{T}$ itself. The primary aim of this section is, given a tensor triangulated category $\mathcal{T}$ with an algebraic model, to produce algebraic models for certain triangulated subquotients of $\mathcal{T}$, as well as for their idempotent completions.

\subsubsection{Monoidal models.}

Let us first recall the notions of Frobenius pair and tensor Frobenius pair.
\begin{dfn}[see \cite{schlichtingnegative}*{Section 3.4}, Definition \ref{dfnfrobpair}]
	A \emph{Frobenius category} is an exact category $\mathcal{E}$ that has enough injective and enough projective objects such that the classes of injective and projective objects in $\mathcal{E}$ coincide. A \emph{Frobenius pair} $\mathbf{E}=(\mathcal{E},\mathcal{E}_0)$ is a strictly full, faithful and exact inclusion of Frobenius categories $\mathcal{E}_0 \hookrightarrow \mathcal{E}$ such that the projective-injective objects of $\mathcal{E}_0$ are mapped to the projective-injective objects of~$\mathcal{E}$.
\end{dfn}

To every Frobenius category, we can associate its \emph{stable category}:
\begin{dfn}
	Let $\mathcal{E}$ be a Frobenius category. Its \emph{stable category} $\underline{\mathcal{E}}$ is the category with the same objects as $\mathcal{E}$ and for two objects $a,b \in \underline{\mathcal{E}}$, we have
	\[\Hom_{\underline{\mathcal{E}}}(a,b) := \Hom_{\mathcal{E}}(a,b)/\mathcal{J}~,\]
	where $\mathcal{J}$ is the subgroup of homomorphisms that factor through a projective-injective object of $\mathcal{E}$.
\end{dfn}
It is well-known that the category $\underline{\mathcal{E}}$ has a natural structure of triangulated category (see e.g.\ \cite{kellerderuse}). The distinguished triangles are given as follows: for each object $a \in \mathcal{E}$, choose a conflation 
\[0 \to a \xrightarrow{\iota_a} ia \xrightarrow{\pi_a} \Sigma a \to 0 ~,\] 
where $ia$ is an injective object of $\mathcal{E}$. Then, for a conflation
\[a \xrightarrow{f} b \xrightarrow{g} c\]
in $\mathcal{E}$, we obtain a \emph{basic distinguished triangle}
\[a \xrightarrow{\overline{f}} b \xrightarrow{\overline{g}} c \xrightarrow{\overline{s}} \Sigma a\] in $\underline{\mathcal{E}}$ from the commutative diagram in $\mathcal{E}$
\[
\xymatrix{
	a \ar[r]^f \ar[d]^{\id} & b \ar[r]^g \ar@{.>}[d]& c \ar[d]^{s}\\
	a \ar[r]^{\iota_a} & ia \ar[r]^{\pi_a}& \Sigma a
}
\]
where the dotted arrow exists by the injectivity of $ia$ and induces $s$. The distinguished triangles in $\underline{\mathcal{E}}$ are given as all sequences 
\[x \to y \to z \to \Sigma x\]
isomorphic in $\underline{\mathcal{E}}$ to basic distinguished triangles.

\begin{lma}
	Given a Frobenius pair $\mathbf{E}=(\mathcal{E},\mathcal{E}_0)$, the inclusion $i:\mathcal{E}_0 \hookrightarrow \mathcal{E}$ induces a fully faithful and exact functor $\underline{i}:\underline{\mathcal{E}_0} \hookrightarrow \underline{\mathcal{E}}$
	\label{lmafrobpairinclusion}
\end{lma}
\begin{proof}
	The functor  $\underline{i}$ is well-defined as the inclusion $i$ maps the projective-injective objects of $\mathcal{E}_0$ to the projective-injective objects of~$\mathcal{E}$. It is full because $i$ is so. In order to check faithfulness,
	let $\overline{f} \in \Hom_{\underline{\mathcal{E}_0}}(a,b)$ and assume $\underline{i}(\overline{f}) = 0$, i.e.\ $i(f)$ factors as $i(a) \to x \to i(b) $ for some projective-injective object $x$ of $\mathcal{E}$. Let $a \hookrightarrow y$ be an inflation in $\mathcal{E}_0$ with $y$ injective. As $i$ preserves projective-injective objects and is exact, we get a commutative diagram
	\[
	\xymatrix{
	i(a) \ar[r]\ar@{^(->}[rd] & x \ar[r] & i(b) \\
	& i(y) \ar@{.>}[u] &
	}
	\]
	where the dotted arrow exists because of the injectivity of $x$. We therefore see that we can write $i(f)$ as a composition $i(a) \to i(y) \to i(b)$. As $i$ was fully faithful, it follows that we can write $f$ as a composition $a \to y \to b$, from which we see that $\overline{f} = 0$, as $y$ was injective.
	
	The exactness of $\underline{i}$ is a direct consequence of the assumption that $i$ is exact and preserves projective-injective objects.
\end{proof}

Lemma \ref{lmafrobpairinclusion} tells us that we can view $\underline{\mathcal{E}_0} $ as a triangulated subcategory of $\underline{\mathcal{E}}$.

\begin{dfn}[see \cite{schlichtingnegative}*{Definition 3.5}]
	The \emph{derived category} of a Frobenius pair $\mathbf{E}=(\mathcal{E},\mathcal{E}_0)$ is defined as the Verdier quotient
	\[\mathrm{D}(\mathbf{E}) := \underline{\mathcal{E}}/\underline{\mathcal{E}_0}~.\]
\end{dfn}

\begin{ex}[see \cite{schlichtingnegative}*{Section 5.3}]
	Let $\mathcal{E}$ be an exact category and consider the category $\mathrm{Ch^b}(\mathcal{E})$ of chain complexes in $\mathcal{E}$. The category $\mathrm{Ch^b}(\mathcal{E})$  can be made into an exact category, if we let the conflations be those short exact sequences of chain complexes that split in every degree (the splittings need not be compatible with the differentials). In that case, $\mathrm{Ch}^b(\mathcal{E})$ becomes a Frobenius category, where the projective-injective objects are given as the contractible complexes. Therefore, $\underline{\mathrm{Ch^b}(\mathcal{E})} = \mathrm{K^b}(\mathcal{E})$, the usual bounded homotopy category of $\mathcal{E}$. If $\mathrm{Ac^b}(\mathcal{E}) \subset \mathrm{Ch^b}(\mathcal{E})$ denotes the full subcategory of complexes homotopy equivalent to an acyclic complex, then $\mathbf{Ch^b}(\mathcal{E}) := (\mathrm{Ch^b}(\mathcal{E}), \mathrm{Ac^b}(\mathcal{E}))$ is a Frobenius pair and we have 
	\[\mathrm{D}(\mathbf{Ch^b}(\mathcal{E})) = \mathrm{K^b}(\mathcal{E})/\underline{\mathrm{Ac^b}(\mathcal{E})} = \Db(\mathcal{E})~,\]
	the usual bounded derived category of $\mathcal{E}$.
\end{ex}

The notion of \emph{tensor Frobenius pair} is a symmetric monoidal variation on the definition of a Frobenius pair.
\begin{dfn}[see Definition \ref{dfntensorfrobpair} ]
A \emph{tensor Frobenius pair} is the datum of a triple $\mathbf{E}=(\mathcal{E},\mathcal{E}_0,\otimes)$ such that:
\begin{enumerate}[label=(\roman*)]
 \item $(\mathcal{E},\mathcal{E}_0)$ is a Frobenius pair.
 \item $(\mathcal{E},\otimes)$ is a symmetric monoidal category.
 \item For all objects $a \in \mathcal{E}$, the functor $a \otimes -$ is exact, i.e. it preserves conflations.
 \item For all objects $a \in \mathcal{E}$, the functor $a \otimes -$ preserves the projective/injective objects of $\mathcal{E}$.
 \item $\mathcal{E}_0$ is a $\otimes$-ideal in $\mathcal{E}$.
 \item For two inflations $a \rightarrowtail a'$ and $b \rightarrowtail b'$, the canonical map 
 \[a' \otimes b \coprod_{a \otimes b} a \otimes b' \to a' \otimes b'\]
 is an inflation. \label{listtensorfrobdefpushprod}
\end{enumerate}
\label{tensorfrobdef}
\end{dfn}

\begin{notn}
	We will omit the symbol $\otimes$ for the tensor product from the notation for a tensor Frobenius pair when there is no danger of confusion.
\end{notn}

\begin{lma}
	Let $\mathbf{A}=(\mathcal{A},\mathcal{A}_0)$ be a tensor Frobenius pair. Then $\mathrm{D}(\mathbf{A})$ inherits the structure of a tensor triangulated category and the localization functor $q: \mathcal{A} \to \mathrm{D}(\mathbf{A})$ is a tensor functor. \label{derivedismonoidal}
\end{lma}
\begin{proof}
	As $\mathcal{A}_0$ is a tensor ideal, the triangulated subcategory $\underline{\mathcal{A}_0}$ is a tensor ideal in $\underline{\mathcal{A}}$ and thus the quotient $\underline{\mathcal{A}}/\underline{\mathcal{A}_0}$ is a tensor triangulated category where the tensor product $\otimes^L$ is induced from the one on $\mathcal{A}$. Indeed, $\otimes^L$ makes $\mathrm{D}(\mathbf{A})$ a symmetric monoidal category, where the associativity, commutativity and unit natural isomorphisms are given as the images of the ones of $(\mathcal{A},\otimes)$ under the functor $\mathcal{A} \to \mathrm{D}(\mathbf{A})$. The functors $a \otimes^L -$ are exact for all objects $a$ of $\mathrm{D}(\mathbf{A})$ since the definition of tensor Frobenius pair guarantees that $a \otimes -$ is a map of Frobenius pairs for all objects $a$ of $\mathcal{A}$. These maps always induce exact functors on the derived categories (cf. \cite{schlichtingnegative}*{Section 3.5}).
\end{proof}

\begin{ex}
	Let $X$ be a non-singular, separated scheme of finite type over a field. Consider the Frobenius pair $(\mathrm{sPerf}(X),\mathrm{asPerf}(X))$, where $\mathrm{sPerf}(X)$ denotes the exact category of strict perfect complexes on $X$ with conflations the degree-wise split ones and $\mathrm{asPerf}(X)$ is the subcategory of acyclic complexes (see Definition \ref{dfnperf}). In Section \ref{theintproduct} we will see that this is a tensor Frobenius pair with respect to the usual tensor product of chain complexes, with derived category $\Dperf(X)$.
\end{ex}

\begin{ex}
	Let $G$ be a finite group, $k$ be a field such that $\mathrm{char}(k)$ divides $|G|$ and let $kG\Mod$ be the category of finitely generated $kG$-modules, which is a Frobenius category as $kG$ is a self-injective ring. Denote by $kG\proj$ the subcategory of projective modules, then $(kG\Mod,kG\proj)$ is a Frobenius pair. It is also a tensor Frobenius pair with respect to the tensor product of modules $\otimes_k$ and its derived category is $kG\stab$, the stable category of the Frobenius category $kG\Mod$.
\end{ex}

\begin{cor}
	Let $\mathbf{A}=(\mathcal{A},\mathcal{A}_0,\otimes)$ be a tensor Frobenius pair and let $\mathcal{J} \subset \mathrm{D}(\mathbf{A})$ be a tensor ideal. Let $\mathcal{B} \subset \mathcal{A}$ be the full subcategory of those objects that become isomorphic to an object of $\mathcal{J}$ after passing to $\mathrm{D}(\mathbf{A})$. Then $\mathbf{B}:=(\mathcal{B},\mathcal{A}_0)$ is a Frobenius pair and $\mathbf{C}:=(\mathcal{A},\mathcal{B},\otimes)$ is a tensor Frobenius pair, with derived categories $\mathrm{D}(\mathbf{B}) = \mathcal{J}$ and $\mathrm{D}(\mathbf{C}) = \mathrm{D}(\mathbf{A})/\mathcal{J}$.
	\label{corsubtfp}
\end{cor}
\begin{proof}
	From \cite{schlichtingnegative}*{Section 5.2}, we already know that $(\mathcal{B},\mathcal{A}_0)$ and $(\mathcal{A},\mathcal{B})$ are Frobenius pairs with corresponding derived categories $\mathcal{J}$ and $\mathrm{D}(\mathbf{A})/\mathcal{J}$. The fact that $\mathbf{C}$ is a \emph{tensor} Frobenius pair follows since the localization functor $\mathcal{A} \to \mathrm{D}(\mathbf{A})$ is a tensor functor and the preimage of a tensor ideal under such a functor is again a tensor ideal.
\end{proof}

\subsubsection{Models for idempotent completion.}
\label{sectmodic} 
If $\mathcal{T} = \mathrm{D}(\mathbf{A})$ for a given tensor Frobenius pair $\mathbf{A} = (\mathcal{A},\mathcal{A}_0)$, we would like to find a tensor Frobenius pair that models the idempotent completion~$\mathcal{T}^{\natural}$. In order to do so, we need the \emph{countable envelope} $\mathrm{C}\mathbf{A} = (\mathrm{C}\mathcal{A},\mathrm{C}\mathcal{A}_0)$, which is a tensor Frobenius pair associated to $\mathbf{A}$ and whose construction is discussed in Appendix \ref{chaptercountable}. The idea is to first embed $\mathrm{D}(\mathbf{A})$ into $\mathrm{D}(\mathrm{C}\mathbf{A})$, the derived category of the countable envelope of $\mathbf{A}$ (see Theorem \ref{thmceistfrobpair}), which has countable copdroducts and is therefore idempotent complete. Then one takes thick closures. Let us give some more details.

The embedding $\mathbf{A} \to \mathrm{C}\mathbf{A}$ (see Remark \ref{remyonedfactorscountable}) induces a fully faithful embedding 
\[\mathrm{D}(\mathbf{A}) \to \mathrm{D}(\mathrm{C}\mathbf{A})\]
(see \cite{schlichtingnegative}*{Proposition 4.4}). In particular we can view $\mathrm{D}(\mathbf{A})$ as a triangulated subcategory of $\mathrm{D}(\mathrm{C}\mathbf{A})$ and consider its thick closure $\overline{\mathrm{D}(\mathbf{A})} \subset \mathrm{D}(\mathrm{C}\mathbf{A})$ which is a triangulated subcategory as well. By \cite{schlichtingnegative}*{Section 5.2}, $\overline{\mathrm{D}(\mathbf{A})}$ admits a Frobenius model $\mathbf{A}^{\natural}$ given as follows: if $\mathcal{B}$ is the full subcategory of $\mathrm{C}\mathcal{A}$ that consists of objects that are isomorphic to objects of $\overline{\mathrm{D}(\mathbf{A})}$ in $\mathrm{D}(\mathrm{C}\mathbf{A})$, then we put $\mathbf{A}^{\natural} = (\mathcal{B},\mathrm{C}\mathcal{A}_0)$.

\begin{lma}
	Assume that $\mathbf{A}$ is a tensor Frobenius pair. Then the Frobenius pair $\mathbf{A}^{\natural}$ is a tensor Frobenius pair, with the tensor structure inherited from the one of~$\mathrm{C}\mathbf{A}$.
	\label{lmaictfp}
\end{lma}
\begin{proof}
	According to Theorem \ref{thmceistfrobpair}, $\mathrm{C}\mathbf{A}$ is naturally a tensor Frobenius pair. The Frobenius pair $\mathbf{A}^{\natural}$ is given as $(\mathcal{B},\mathcal{A}_0)$, where $\mathcal{B}$ is the full subcategory of $\mathrm{C}\mathcal{A}$ that consists of objects that are isomorphic to objects of $\overline{\mathrm{D}(\mathbf{A})}$ in $\mathrm{D}(\mathrm{C}\mathbf{A})$. From this perspective, it is clear that all we have to prove is that $\mathcal{B}$ is closed under taking $\otimes_{\mathrm{C}\mathcal{A}}$-products.
	
	To do this, notice that by Proposition \ref{propcetensex}, the embedding $\mathrm{D}(\mathbf{A}) \to \mathrm{D}(\mathrm{C}\mathbf{A})$ preserves tensor products, and therefore $\mathrm{D}(\mathbf{A})$ is closed under $\otimes_{\mathrm{C}\mathcal{A}}$-products when we consider it as a triangulated subcategory of $\mathrm{D}(\mathrm{C}\mathbf{A})$. Now, take two objects $A,B$ of $\mathcal{B} \subset \mathcal{A}$ and denote by $L:\mathrm{C}\mathcal{A} \to \mathrm{D}(\mathrm{C}\mathbf{A})$ the localization functor given as the composition \[\mathrm{C}\mathcal{A} \to \underline{\mathrm{C}\mathcal{A}} \to \underline{\mathrm{C}\mathcal{A}}/\underline{\mathrm{C}\mathcal{A}_0} = \mathrm{D}(\mathrm{C}\mathbf{A})~.\] 
	The functor $L$ preserves tensor products since both functors in the composition do. By definition of thick closure there exist two objects $A',B' \in \mathcal{B}$ such that $L(A) \oplus L(A') \in \mathrm{D}(\mathbf{A})$ and $L(B) \oplus L(B) \in \mathrm{D}(\mathbf{A})$. Thus
	\begin{align*}
	&\left(L(A) \oplus L(A')\right) \otimes_{\mathrm{D}(\mathcal{F}\mathbf{A})} \left(L(B) \oplus L(B')\right) \cong\\ 
	&\cong \left(L(A)  \otimes_{\mathrm{D}(\mathcal{F}\mathbf{A})} L(B)\right) \oplus \left(L(A)  \otimes_{\mathrm{D}(\mathcal{F}\mathbf{A})} L(B')\right) \oplus \left(L(B)  \otimes_{\mathrm{D}(\mathcal{F}\mathbf{A})} L(A')\right) \\
	& \phantom{\cong} \oplus \left(L(B)  \otimes_{\mathrm{D}(\mathcal{F}\mathbf{A})} L(B')\right)
	\end{align*}
	which shows that $L(A)  \otimes_{\mathrm{D}(\mathrm{C}\mathbf{A})} L(B)= L(A \otimes_{\mathrm{C}\mathcal{A}} B)$ is isomorphic to a direct summand of an object in $\mathrm{D}(\mathbf{A})$ and proves that $A \otimes_{\mathrm{C}\mathcal{A}} B \in \mathcal{B}$.
\end{proof}

\begin{lma}
	The category $\mathrm{D}\left(\mathbf{A}^{\natural}\right)$ realizes the idempotent completion $\left(\mathrm{D}(\mathbf{A})\right)^{\natural}$ as a tensor triangulated category.
	\label{idempmodel}
\end{lma}

\begin{proof}
	This follows as $\mathrm{D}(\mathrm{C}\mathbf{A})$ is idempotent complete (since it has countable coproducts by \cite{schlichtingnegative}*{Proposition 4.4}) and $\mathrm{D}\left(\mathbf{A}^{\natural}\right)$ is the thick closure of $\mathrm{D}(\mathbf{A})$ in $\mathrm{D}(\mathrm{C}\mathbf{A})$. The equivalence is explicitly given by sending a pair $(a,e)$ in $\mathrm{D}\left(\mathbf{A}\right)^{\natural}$, with $a$  an object of $\mathrm{D}\left(\mathbf{A}\right)$ and $e: a \to a$ an idempotent endomorphism, to $\im(e) \in \mathrm{D}\left(\mathbf{A}^{\natural}\right)$. We see that this equivalence preserves the tensor product, as the embedding $\mathrm{D}\left(\mathbf{A}\right) \to \mathrm{D}\left(\mathbf{A}^{\natural}\right)$ preserves tensor products by Proposition \ref{propcetensex}.
\end{proof}

\begin{lma}
	The assignment $\mathbf{A} \mapsto \mathbf{A}^{\natural}$ is functorial for maps of Frobenius pairs.
	\label{lmaidemfunc}
\end{lma}
\begin{proof}
	The assignment $\mathbf{A} \mapsto \mathrm{C}\mathbf{A}$ is functorial (see \cite{schlichtingnegative}*{Definition 4.3}) and so a map of Frobenius pairs $m: \mathbf{A} \to \mathbf{B}$ gives a map $\mathrm{C}m: \mathrm{C}\mathbf{A} \to \mathrm{C}\mathbf{B}$. By the additivity of $\mathrm{C}m$ it follows that its restriction to $\mathbf{A}^{\natural}$ maps into $\mathbf{B}^{\natural}$ which proves the lemma.
\end{proof}
As a consequence of Lemma \ref{idempmodel}, we now have a Frobenius model for $\left(\mathrm{D}(\mathbf{A})\right)^{\natural}$ at our disposal.

\subsection{Higher and negative algebraic $\mathrm{K}$-theory of a Frobenius pair} 
Let $\mathbf{A} = (\mathcal{A},\mathcal{A}_0)$ be a Frobenius pair. In \cite{schlichtingnegative}*{Section 11}, Schlichting defines a $\mathrm{K}$-theory spectrum $\mathbb{K}(\mathbf{A})$ for $\mathbf{A}$ that we will use in the following. The associated $\mathrm{K}$-groups of $\mathbf{A}$ are given as follows (see \cite{schlichtingnegative}*{Theorem 11.7}): 
\begin{itemize}
	\item For $i>0$, the groups $\mathbb{K}_i(\mathbf{A})$ are the Waldhausen $\mathrm{K}$-groups of $\mathbf{A}$. That is, we make $\mathcal{A}$ into a category with cofibrations and weak equivalences by declaring the cofibrations to be the inflations of $\mathcal{A}$ and the weak equivalences those morphisms that become isomorphisms in $\mathrm{D}(\mathbf{A})$. Then $\mathbb{K}_i(\mathbf{A})$ is the $i$-th Waldhausen $\mathrm{K}$-group $\mathrm{K}_i^W(\mathcal{A})$ of the category with cofibrations and weak equivalences $\mathcal{A}$.
	\item $\mathbb{K}_0(\mathbf{A}) = \mathrm{K}_0\left(\mathrm{D}(\mathbf{A})^{\natural}\right)$.
	\item For $i<0$ one defines $\mathbb{K}_{i}(\mathbf{A})$ as follows: Let $\mathrm{S}_0\mathbf{A}$ denote the full subcategory of $\mathrm{C}\mathcal{A}$ consisting of all objects in the kernel of the Verdier quotient functor 
	\[\mathrm{D}(\mathrm{C}\mathbf{A}) \to \mathrm{D}(\mathrm{C}\mathbf{A})/\mathrm{D}(\mathbf{A})~.\] 
	The suspension $\mathrm{S}\mathbf{A}$ of $\mathbf{A}$ is defined as the Frobenius pair $(\mathrm{C}\mathcal{A}, \mathrm{S}_0\mathbf{A})$, and for $n \geq 1$, $\mathrm{S}^n\mathbf{A}$ denotes the Frobenius pair obtained from $\mathbf{A}$ by applying the suspension construction $n$ times. For $i<0$, Schlichting (see \cite{schlichtingnegative}*{Definition 4.7}) defines 
	\[\mathbb{K}_{i}(\mathbf{A}) := \Kzero(\mathrm{S}^{-i}\mathbf{A})~.\] 
\end{itemize}

One then obtains long exact localization sequences. Let
\[\mathbf{B} \to \mathbf{A} \to \mathbf{C}\] 
be an exact sequence of Frobenius pairs, i.e.\ one such that the induced sequence
\[\mathrm{D}(\mathbf{B}) \to \mathrm{D}(\mathbf{A}) \to \mathrm{D}(\mathbf{C})\] 
is exact up to factors: the composition is zero, the functor $\mathrm{D}(\mathbf{B}) \to \mathrm{D}(\mathbf{A})$ is fully faithful and the induced functor 
\[\mathrm{D}(\mathbf{A})/\mathrm{D}(\mathbf{B}) \to \mathrm{D}(\mathbf{C})\] 
is cofinal. Then we obtain a long exact localization sequence
\[\cdots \to \mathbb{K}_{p}(\mathbf{B}) \to \mathbb{K}_{p}(\mathbf{A}) \to \mathbb{K}_{p}(\mathbf{C}) \to \mathbb{K}_{p-1}(\mathbf{B}) \to \cdots \]
for all $p \in \mathbb{Z}$ (see \cite{schlichtingnegative}*{Theorem 11.10}). 

\begin{rem}
	Assume that $\mathcal{T} = \mathrm{D}(\mathbf{A})$ for a \emph{tensor} Frobenius pair $\mathbf{A}$, such that $\mathcal{T}$ is a tensor triangulated category. Let $\mathcal{J} \subset \mathcal{T}$ be a tensor ideal. Corollary \ref{corsubtfp} and Lemmas \ref{lmaictfp} and \ref{idempmodel} provide models for $\mathbf{B}$ and $\mathbf{C}$ for $\mathcal{J}$ and $(\mathcal{T}/\mathcal{J})^{\natural}$, respectively. The sequence of Frobenius pairs
	\[\mathbf{B} \to \mathbf{A} \to \mathbf{C}\]
	induces the sequence of derived categories
	\[\mathcal{J} \to \mathcal{T} \to (\mathcal{T}/\mathcal{J})^{\natural}\]
	which is exact up to factors. This gives us a long exact sequence in $\mathrm{K}$-theory
	\[\cdots \to \mathbb{K}_{p}(\mathbf{B}) \to \mathbb{K}_{p}(\mathbf{A}) \to \mathbb{K}_{p}\left(\mathbf{C}\right) \to \mathbb{K}_{p-1}(\mathbf{B}) \to \cdots ~.\]
\end{rem}

\begin{rem}
	An application of the localization sequence implies the following: if we are given two Frobenius pairs with equivalent derived categories, and the equivalence is induced by a functor on the level of Frobenius pairs (which need \emph{not} be an equivalence), then the $\mathrm{K}$-groups arising from the two different models will be isomorphic. This is why we informally think of the $\mathrm{K}$-theory of the Frobenius pair $\mathbf{A}$ as the $\mathrm{K}$-theory of the triangulated category $\mathrm{D}(\mathbf{A})$. One must be careful though: it is not true that any model of $\mathrm{D}(\mathbf{A})$ yields the same $\mathrm{K}$-theory (see \cite{schlichtingknote}).
\end{rem}

\section{$\mathrm{K}$-theory sheaves on the spectrum}
Let us start by proving a basic but useful lemma. Notice first that any dimension function on a tensor triangulated category $\mathcal{T}$ induces a dimension function on $\mathcal{T}^{\natural}$, as the inclusion $\mathcal{T} \hookrightarrow \mathcal{T}^{\natural}$ induces a homeomorphism of spectra (see Example \ref{exidcomptt}).
\begin{lma}
	Let $\mathcal{T}$ be an essentially small tensor triangulated category that is equipped with a dimension function $\dim$. Then for all $l \in \mathbb{Z}$, the subcategory
	\[\mathcal{T}_{(l)} \subset \left(\mathcal{T}^{\natural}\right)_{(l)}\]
	is dense. Therefore the inclusion induces an equivalence
	\[\left(\mathcal{T}_{(l)}\right)^{\natural} \cong \left(\mathcal{T}^{\natural}\right)_{(l)}~.\]
	\label{lmasubidemcom}
\end{lma}
\begin{proof}
	As $\mathcal{T}$ is dense in $\mathcal{T}^{\natural}$, for every object $a \in \mathcal{T}^{\natural}$, $a \oplus \Sigma(a) \in \mathcal{T}$. Indeed, this follows by Thomason's classification of dense subcategories (see \cite{thomasonclassification}) which gives
	\[\mathcal{T} = \left\lbrace a \in \mathcal{T}^{\natural}: [a] \in \Kzero(\mathcal{T}) \subset \Kzero\left(\mathcal{T}^{\natural}\right)\right\rbrace~.\]
	Given $b \in \left(\mathcal{T}^{\natural}\right)_{(l)}$, we have $\Sigma(b)\in \left(\mathcal{T}^{\natural}\right)_{(l)}$ as well and by our previous argument $b \oplus \Sigma(b) \in \mathcal{T}$. As 
	\[\dim(\supp(b \oplus \Sigma(b))) = \dim(\supp(b) \cup \supp(\Sigma(b))) = \dim(\supp(b)) \leq l~,\]
	it follows that $b \oplus \Sigma(b) \in \mathcal{T}_{(l)}$. This shows that every object of $\left(\mathcal{T}^{\natural}\right)_{(l)}$ is a direct summand of an object of $\mathcal{T}_{(l)}$ and therefore proves the claim.
\end{proof}

Before we define $\mathrm{K}$-theory sheaves on $\Spc(\mathcal{T})$ we fix some assumptions on $\mathcal{T}$ that we will need for the rest of the article.
\begin{conv}
	From now on we fix a tensor Frobenius pair $\mathbf{A} = (\mathcal{A},\mathcal{A}_0,\otimes)$ and let $\mathcal{T} = \mathrm{D}(\mathbf{A})$. We assume $\mathcal{T}$ to be essentially small, rigid, equipped with a dimension function $\dim$ and such that $\Spc(\mathcal{T})$ is noetherian. We also implicitly assume that whenever we look at a Verdier quotient $\mathcal{T}/\mathcal{T}_Z$ of $\mathcal{T}$ by a thick $\otimes$-ideal $\mathcal{T}_Z$, the category $\mathcal{T}/\mathcal{T}_Z$ is equipped with the dimension function $\dim|_U$ as in Construction \ref{constrresdim}. In this situation, we will denote by $\mathcal{T}_U$ the category $(\mathcal{T}/\mathcal{T}_Z)^{\natural}$.
	\label{convbasicass}
\end{conv}

\begin{dfn}
	For any $p \in \mathbb{Z}_{\geq 0}, l \in \mathbb{Z}$, the sheaf $\mathscr{K}^{(l)}_p$ on $\Spc{\mathcal{T}}$ is defined as the sheaf associated to the presheaf
	\[U \mapsto \mathbb{K}_p\left((\mathbf{A}_U)_{(l)}\right)\]
	for an open $U \subset \Spc(\mathcal{T})$ with complement $Z$. Here, $(\mathbf{A}_U)_{(l)}$ is the Frobenius pair obtained from $\mathbf{A}$ by subsequently taking models for the Verdier quotient $\mathcal{T}/\mathcal{T}_Z$, then for the triangulated subcategory $\left(\mathcal{T}/\mathcal{T}_Z\right)_{(l)}$ and finally for the idempotent completion $\left(\left(\mathcal{T}/\mathcal{T}_Z\right)_{(l)}\right)^{\natural} \cong \left(\mathcal{T}_U\right)_{(l)}$ (by Lemma \ref{lmasubidemcom}), as described in Section \ref{sectionmodels}. By construction, we then have 
	\[\mathrm{D}((\mathbf{A}_U)_{(l)}) = \left(\mathcal{T}_U\right)_{(l)}~.\]
	The restriction map $\mathscr{K}^{(l)}_p(U) \to \mathscr{K}^{(l)}_p(V)$ for two opens $V \subset U \subset \Spc(\mathcal{T})$ with complements $W \supset Z$ respectively is induced in the following way: the Frobenius pair that models $\left(\mathcal{T}/\mathcal{T}_Z\right)_{(l)}$ is given by $(\mathcal{A}^U_{(l)},\mathcal{A}_Z)$, where $\mathcal{A}_Z$ is the full subcategory of $\mathcal{A}$ consisting of those objects that become isomorphic to objects of $\mathcal{T}_Z$ in $\mathrm{D}(\mathbf{A})$ and $\mathcal{A}^U_{(l)}$ is the full subcategory of $\mathcal{A}$ consisting of those objects that become isomorphic to objects of $\left(\mathcal{T}/\mathcal{T}_Z\right)_{(l)}$ in $\mathrm{D}((\mathcal{A},\mathcal{A}_Z)) =  \mathcal{T}/\mathcal{T}_Z$. Using Lemma \ref{lmareshasreldim0}, we see that there is a map of Frobenius pairs
	\[(\mathcal{A}^U_{(l)},\mathcal{A}_Z) \to (\mathcal{A}^V_{(l)},\mathcal{A}_W)\]
	given by inclusion. After applying idempotent completion as in Lemma \ref{lmaidemfunc} we obtain a map of Frobenius pairs 
	\[(\mathbf{A}_U)_{(l)} \to (\mathbf{A}_V)_{(l)}\]
	which induces the restriction map.
	
	Similarly for any $p \in \mathbb{Z}_{\geq 0}, l \in \mathbb{Z}$, we define the sheaves $\mathscr{K}^{(l/l-1)}_p$ on $\Spc(\mathcal{T})$ as the sheaves associated to the presheaves
	\[U \mapsto \mathbb{K}_p\left((\mathbf{A}_U)_{(l)/(l-1)}\right)\]
	for an open $U \subset \Spc(\mathcal{T})$ with complement $Z$. Here, $(\mathbf{A}_U)_{(l)/(l-1)}$ is the Frobenius pair associated to the subquotient $\left((\mathcal{T}_U)_{(l)}/(\mathcal{T}_U)_{(l-1)}\right)^{\natural}$ of $\mathcal{T}$, given as $(\mathcal{A}^U_{(l)},\mathcal{A}^U_{(l-1)})^{\natural}$. By construction, \[\mathrm{D}\left(\left(\mathcal{A}^U_{(l)},\mathcal{A}^U_{(l-1)}\right)\right) \cong \left(\mathcal{T}/\mathcal{T}_Z\right)_{(l)}/\left(\mathcal{T}/\mathcal{T}_Z\right)_{(l-1)}\]
	and thus we indeed have
	\begin{align*}
	\mathrm{D}\left(\left(\mathcal{A}^U_{(l)},\mathcal{A}^U_{(l-1)}\right)^{\natural}\right) &\cong \left(\left(\mathcal{T}/\mathcal{T}_Z\right)_{(l)}/\left(\mathcal{T}/\mathcal{T}_Z\right)_{(l-1)} \right)^{\natural}\\ 
	&\cong \left(\left(\mathcal{T}/\mathcal{T}_Z\right)_{(l)}^{\natural}/\left(\mathcal{T}/\mathcal{T}_Z\right)_{(l-1)}^{\natural} \right)^{\natural} \\
	&\cong \left((\mathcal{T}_U)_{(l)}/(\mathcal{T}_U)_{(l-1)}\right)^{\natural}
	\end{align*}
	by \cite{balmerfiltrations}*{Proposition 1.13} and Lemma \ref{lmasubidemcom}. For an open $V \subset U$, there is a map of Frobenius pairs
	\[\left(\mathcal{A}^U_{(l)},\mathcal{A}^U_{(l-1)}\right) \to \left(\mathcal{A}^V_{(l)},\mathcal{A}^V_{(l-1)}\right)\]
	given by inclusion. Again, after applying idempotent completion as in Lemma \ref{lmaidemfunc} we obtain a map of Frobenius pairs 
	\[(\mathbf{A}_U)_{(l)/(l-1)} \to (\mathbf{A}_V)_{(l)/(l-1)}\]
	which induces the restriction map.
	\label{ksheafdef}
\end{dfn}

The next result is a key instrument for the constructions of the following sections, as it shows that we can use the sheaves $\mathscr{K}^{(l/l-1)}_p$ to calculate cohomology.
\begin{prop}
	For any $p \in \mathbb{Z}_{\geq 0}, l \in \mathbb{Z}$, the sheaves $\mathscr{K}^{(l/l-1)}_p$ are flasque.
	\label{propflasque}
\end{prop}
\begin{proof}
	We show that the presheaf
	\[U \mapsto \mathbb{K}_p\left((\mathbf{A}_U)_{(l)/(l-1)}\right)\]
	is already a sheaf and that it is flasque. The main point here is that the equivalence 
	\begin{equation}
	\left((\mathcal{T}_U)_{(l)}/(\mathcal{T}_U)_{(l-1)}\right)^{\natural} \cong \coprod_{\substack{Q \in U \\ \dim(Q) = l}} \Min(\mathcal{T}_Q)
	\label{eqbaldecomp}
	\end{equation}
	from Theorem \ref{thmfiltdecomp} is induced on the level of Frobenius models. 
		
	For $Q \in U, \dim(P) = l$, the Frobenius pair associated to $\Min(\mathcal{T}_Q)$ is constructed as follows: we let $\mathcal{A}_Q \subset \mathcal{A}$ be the full subcategory of those objects becoming isomorphic to objects of $Q \subset \mathcal{T}$ in $\mathrm{D}(\mathbf{A}) = \mathcal{T}$. Let $\mathcal{A}_{\mathrm{Min}} \subset \mathcal{A}$ be the full subcategory of objects that become isomorphic to objects with minimal support in $\mathrm{D}\left((\mathcal{A},\mathcal{A}_Q)\right) = \mathcal{T}/Q$. The Frobenius model we use for $\Min(\mathcal{T}_Q)$ is then given as $(\mathcal{A}_{\mathrm{Min}},\mathcal{A}_Q)^{\natural}$ which we will denote by $\mathbf{Min}_Q$. Indeed, by construction we have
	\[\mathrm{D}\left(\mathbf{Min}_Q\right) = \mathrm{D}\left((\mathcal{A}_{\mathrm{Min}},\mathcal{A}_Q)^{\natural}\right) \cong (\mathrm{Min}(\mathcal{T}/Q))^{\natural} \cong \mathrm{Min}(\mathcal{T}_Q) ~.\]
	The last equivalence follows by Lemma \ref{lmasubidemcom} as $\mathrm{Min}(\mathcal{T}/Q) = (\mathcal{T}/Q)_{(n)}$, where $n \in \mathbb{Z}$ is the dimension of the unique closed point of $\mathcal{T}/Q$.
	
	There is an inclusion $\mathcal{A}^U_{(l-1)} \subset \mathcal{A}_Q$ (see Definition \ref{ksheafdef}) by \cite{balmerfiltrations}*{Prop.~3.21}. We also have $\mathcal{A}^U_{(l)} \subset \mathcal{A}_{\mathrm{Min}}$ which implies that we get a map of Frobenius pairs
	\[(\mathcal{A}^U_{(l)},\mathcal{A}^U_{(l-1)}) \to (\mathcal{A}_{\mathrm{Min}},\mathcal{A}_Q)\]
	for all $Q \in U$, given by inclusion. After idempotent completion we obtain maps 
	\[(\mathbf{A}_U)_{(l)/(l-1)} \to \mathbf{Min}_Q\]
	and the sum of these maps for all $P \in U$
	\[\epsilon_U: (\mathbf{A}_U)_{(l)/(l-1)} \to \coprod_{\substack{Q \in U \\ \dim(Q) = l}} \mathbf{Min}_Q \]
	induces the equivalence (\ref{eqbaldecomp}) on the derived categories.
	
	As a consequence, we see that the sheaf $\mathscr{K}^{(l/l-1)}_p$ is given as the sheafification of the presheaf
	\[U \mapsto \coprod_{\substack{Q \in U \\ \dim(Q) = l}} \mathbb{K}_p\left(\mathbf{Min}_Q\right)~.\]
	Now, for two opens $V \subset U$ consider the diagram
	\[
	\xymatrix@=4em{
		(\mathbf{A}_U)_{(l)/(l-1)} \ar[d]^{\epsilon_U} \ar[r]^{\mathrm{res}} & (\mathbf{A}_V)_{(l)/(l-1)} \ar[d]^{\epsilon_V} \\
		\displaystyle\coprod\limits_{\substack{Q \in U \\ \dim(Q) = l}} \mathbf{Min}_Q \ar@<10pt>[r]^{\pi} & \displaystyle\coprod\limits_{\substack{Q \in V \\ \dim(Q) = l}} \mathbf{Min}_Q
	}
	\]
	where $\mathrm{res}$ is the restriction functor from Definition \ref{ksheafdef} and $\pi$ is the canonical projection. One checks that this square is commutative. The maps $\epsilon_U$ and $\epsilon_V$ become equivalences on the corresponding derived categories and therefore $\mathbb{K}_p(\epsilon_U), \mathbb{K}_p(\epsilon_V)$ become isomorphisms and the square commutes after applying $\mathbb{K}_p(-)$. It follows that the restriction maps of the presheaf
	\[U \mapsto \coprod_{\substack{Q \in U \\ \dim(Q) = l}} \mathbb{K}_p\left(\mathbf{Min}_Q\right)\]
	are given as the canonical projections.
	
	We now show that this presheaf is already a sheaf (and will therefore coincide with $\mathscr{K}^{(l/l-1)}_p$): from the nature of the restriction maps, it is clear that an element of the group $\mathbb{K}_p\left(\mathbf{A}_U)_{(l)/(l-1)}\right)$ with trivial restriction to an open cover must be trivial on $U$. Furthermore, if we are given an open covering $U = \bigcup\limits_{i \in I} V_i$ and $s_i \in \mathbb{K}_p\left((\mathbf{A}_{V_i})_{(l)/(l-1)}\right)$ with compatible restrictions to the mutual intersections, we can glue them together to an element $s \in \mathbb{K}_p\left(\mathbf{A}_U)_{(l)/(l-1)}\right)$: from the $s_i$ we know what the germ $s_P$ of $s$ at $P$ should be for every $P \in U$. In order to check that there are only finitely many non-zero $s_P$'s, we use that $\Spc(\mathcal{T})$ was assumed to be noetherian and thus finitely many $V_{i_1}, \ldots, V_{i_n}$ suffice to cover $U$. By definition, $(s_{i_j})_P = 0$ for all but finitely many $P \in V_{i_j}$ for $j =1 ,\ldots, n$. This implies that $s_P = 0$ for all but finitely many $P \in U$ and thus $s \in \mathbb{K}_p\left(\mathbf{A}_U)_{(l)/(l-1)}\right)$ as desired.
	
	The flasqueness of $\mathscr{K}^{(l/l-1)}_p$ now follows directly, as its restriction maps coincide with those of the presheaf, and these are clearly surjective.
\end{proof}

\section{Triangulated Gersten conjecture and triangulated Bloch formula}
\subsection{The triangulated Gersten conjecture} 
We stick to our assumptions from Convention \ref{convbasicass}. For any $l \in \mathbb{Z}$ and $U \subset \Spc(\mathcal{T})$ we have a sequence of Frobenius pairs
\[(\mathbf{A}_U)_{(l-1)} \to (\mathbf{A}_U)_{(l)} \to (\mathbf{A}_U)_{(l)/(l-1)}\]
which induces a sequence of tensor triangulated categories
\[(\mathcal{T}_U)_{(l-1)} \hookrightarrow (\mathcal{T}_U)_{(l)} \to \left((\mathcal{T}_U)_{(l)}/(\mathcal{T}_U)_{(l-1)}\right)^{\natural} \]
that is exact up to factors. Therefore we obtain localization sequences
\[\cdots \to \mathbb{K}_{p}\left((\mathbf{A}_U)_{(l)}\right) \to \mathbb{K}_{p}\left((\mathbf{A}_U)_{(l)/(l-1)}\right) \to \mathbb{K}_{p-1}\left((\mathbf{A}_U)_{(l-1)}\right) \to \cdots \]
which, by applying sheafification, give us a long exact sequence of sheaves
\begin{equation}
\cdots \to \mathscr{K}^{(l-1)}_{p} \to \mathscr{K}^{(l)}_{p} \to \mathscr{K}^{(l/l-1)}_{p} \to \mathscr{K}^{(l-1)}_{p-1} \to \cdots ~.
\label{longexactgersten}
\end{equation}

\begin{dfn}
	We say that \emph{the triangulated Gersten conjecture holds for the Frobenius pair $\mathbf{A}$ (see Convention \ref{convbasicass}) in bidegree $(l,p)$} for $(l,p)\in \mathbb{Z}^2$ if in the above long exact sequence (\ref{longexactgersten}), the map $\mathscr{K}^{(l-1)}_{p} \to \mathscr{K}^{(l)}_{p}$ vanishes.
	\label{gersten}
\end{dfn}

\begin{rem}
	Whether the triangulated Gersten conjecture holds for $\mathbf{A}$ might depend on the choice of dimension function for $\mathcal{T}$.
\end{rem}

\begin{rem}
	As we will see in Lemma \ref{gerstenexample}, the triangulated Gersten conjecture can be viewed as a generalization of the usual Gersten conjecture from algebraic $\mathrm{K}$-theory. Let us recall its statement:
	\begin{announceconj}[Gersten]
		Let $X$ be the spectrum of a regular local ring $R$. Let $\mathrm{M}_l(X)$ denote the category of coherent sheaves on $X$ with codimension of support~$\geq l$ with associated Quillen $\mathrm{K}$-groups $\mathrm{K}_p(\mathrm{M}_l(X))$ for~$p \geq 0$. Then the maps 
		\[\mathrm{K}_p(\mathrm{M}_{l+1}(X)) \to \mathrm{K}_p(\mathrm{M}_l(X))\]
		induced for all $p \geq 0$ by the inclusion $\mathrm{M}_{l+1}(X) \to \mathrm{M}_{l}(X)$ vanish.
	\end{announceconj}
	The conjecture was proved by Quillen in \cite{quillenhigher} for the case that $R$ is a finitely generated algebra over a field, and later Panin \cite{panin} removed the finite generation hypothesis. Quillen uses his result in \cite{quillenhigher} to prove the \emph{Bloch formula}, which identifies the Chow groups of a non-singular variety $X$ with certain cohomology groups of $\mathrm{K}$-theory sheaves on $X$. We will use the triangulated Gersten conjecture for a similar purpose in Theorem \ref{bloch}. If we pass to sheafified Quillen $\mathrm{G}$-theory on $X$, the conjecture is equivalent to the vanishing of maps in the localization sequence associated to the coniveau filtration, similar to the requirement in Definition \ref{gersten}.
\end{rem}

\begin{rem}
If $\mathbf{A}$ satisfies the triangulated Gersten conjecture in bidegrees $(l,p)$ and ${(l,p-1)}$, then the long exact sequence (\ref{longexactgersten}) contains the short exact sequence
\begin{equation}
0 \to \mathscr{K}^{(l)}_{p} \to \mathscr{K}^{(l/l-1)}_{p} \to \mathscr{K}^{(l-1)}_{p-1} \to 0 ~.
\label{ses}
\end{equation}
\end{rem}

\subsection{The triangulated Bloch formula}
For any essentially small tensor triangulated category $\mathcal{L}$ equipped with a dimension function and $l \in \mathbb{Z}$, we can define sheaves of Grothendieck groups on $\Spc(\mathcal{L})$ as follows: let $\mathcal{F}^{l}(\mathcal{L})$ denote the sheaf associated to the presheaf
\[U \mapsto \Kzero\left((\mathcal{L}_U)_{(l)}\right))\]
and let $\mathcal{F}^{l/l-1}(\mathcal{L})$ denote the sheaf
\[U \mapsto \Kzero\left(\left((\mathcal{L}_U)_{(l)}/(\mathcal{L}_U)_{(l-1)}\right)^{\natural}\right)\]
so that we have $\mathcal{F}^{l}(\mathrm{D}(\mathbf{A})) = \mathscr{K}^{(l)}_0$ and $\mathcal{F}^{l/l-1}(\mathrm{D}(\mathbf{A})) = \mathscr{K}^{(l/l-1)}_0$ as special cases (see Definition~\ref{ksheafdef}). Note that for $\mathcal{F}^{l/l-1}(\mathcal{L})$, we don't need to sheafify by Proposition~\ref{propflasque}. There is also a map of sheaves
\begin{equation}
\beta: \mathcal{F}^{l}(\mathcal{L}) \to \mathcal{F}^{l/l-1}(\mathcal{L})
\end{equation}
which is obtained as the sheafification of a map of presheaves $\beta'$ induced by the composition of the Verdier localization functor and the inclusion into the idempotent completion: 
\[\beta'(U):  \Kzero\left((\mathcal{L}_U)_{(l)}\right)) \rightarrow  \Kzero\left((\mathcal{L}_U)_{(l)}/(\mathcal{L}_U)_{(l-1)}\right) \hookrightarrow \Kzero\left(\left((\mathcal{L}_U)_{(l)}/(\mathcal{L}_U)_{(l-1)}\right)^{\natural}\right) ~.\]
For $\mathcal{L} = \mathrm{D}(\mathbf{A})$, the map $\beta$ is the one of the localization sequence (\ref{longexactgersten}). We will be interested in the group of global sections
\begin{equation}
\Gamma(\im(\beta)) \subset \Gamma\left(\mathcal{F}^{l/l-1}(\mathcal{L})\right) = \Kzero\left((\mathcal{L}_{(l)}/\mathcal{L}_{(l-1)})^{\natural}\right) = \Cyc^{\Delta}_l(\mathcal{L})~,
\label{eqcyccontainment}
\end{equation}
where $\Cyc^{\Delta}_l(\mathcal{L})$ is the dimension $l$ tensor triangular cycle group of $\mathcal{L}$ from Definition \ref{defcycle}. The image of the map of presheaves $\beta'$ on the level of global sections is the subgroup
\[\Gamma(\im(\beta')) = \Kzero\left(\mathcal{L}_{(l)}^{\natural}/\mathcal{L}^{\natural}_{(l-1)}\right) \subset \Kzero\left((\mathcal{L}_{(l)}/\mathcal{L}_{(l-1)})^{\natural}\right)~.\]
As the presheaf $\im(\beta')$ is separated (it is, after all, a sub-presheaf of a sheaf), the natural map $\im(\beta') \to \im(\beta)$ from presheaf to sheafification is injective and thus we have an inclusion
\begin{equation}
j: \Gamma(\im(\beta')) = \Kzero\left(\mathcal{L}_{(l)}^{\natural}/\mathcal{L}^{\natural}_{(l-1)}\right) \hookrightarrow \Gamma(\im(\beta))
\label{eqinclidcomp}
\end{equation}
as well. Let $i: \Kzero(\mathcal{L}_{(l)}^{\natural}) \to \Kzero(\mathcal{L}_{(l+1)}^{\natural})$ be the map induced by the inclusion and $\phi: \Kzero(\mathcal{L}_{(l)}^{\natural}) \to \Kzero(\mathcal{L}_{(l)}^{\natural}/\mathcal{L}^{\natural}_{(l-1)})$ be the map induced by the Verdier quotient functor.

\begin{dfn}
	The \emph{$l$-dimensional $\cap$-cycle group} of $\mathcal{L}$ is defined as the group
	\[\prescript{}{\cap}\Cyc^{\Delta}_l(\mathcal{L}) := \Gamma(\im(\beta)) \subset  \Cyc^{\Delta}_l(\mathcal{L}) ~.\]
	The \emph{$l$-dimensional $\cap$-Chow group} $\mathcal{L}$ is defined as the quotient
	\[\prescript{}{\cap}\Chow^{\Delta}_l(\mathcal{L}) := \prescript{}{\cap}\Cyc^{\Delta}_l(\mathcal{L})/j \circ \phi(\ker(i))~.\]
	\label{dfnaltchow2}
\end{dfn}

\begin{rem}
	We will see in Theorem \ref{bloch} that these $\cap$-Chow groups show up in the cohomology of the sheaf $\mathscr{K}^{(0)}_p$ (see Definition \ref{ksheafdef}). 
	From Definition \ref{defchow}, it also follows that
	\[\prescript{}{\cap}\Chow^{\Delta}_l(\mathcal{L}) \subset \Chow^{\Delta}_{l}(\mathcal{L})~.\]
	When $\mathcal{L}_{(l)}^{\natural}/\mathcal{L}_{(l-1)}^{\natural}$ is idempotent complete already, it follows from (\ref{eqinclidcomp}) that
	\[\prescript{}{\cap}\Cyc^{\Delta}_l(\mathcal{L}) = \Cyc^{\Delta}_l(\mathcal{L}) \quad \text{and} \quad \prescript{}{\cap}\Chow^{\Delta}_l(\mathcal{L}) = \Chow^{\Delta}_{l}(\mathcal{L})~.\]
	\label{remcapnormcomp}
\end{rem}

\begin{ex}
	Let $X$ be a non-singular separated scheme of finite type over a field and $\mathcal{L} = \Dperf(X)$, the derived category of perfect complexes equipped with the opposite of the codimension of support as a dimension function. By Theorem \ref{agreementthm} we have that $\Cyc^{\Delta}_{-n}(\mathcal{L}) \cong \Cyc^{n}(X)$ and $\Chow^{\Delta}_{-n}(\mathcal{L}) \cong \Chow^{n}(X)$ for all $n \in \mathbb{Z}$. In this case we also have isomorphisms $\prescript{}{\cap}\Cyc^{\Delta}_{-n}(\mathcal{L}) \cong \Cyc^{\Delta}_{-n}(\mathcal{L})$ and $\prescript{}{\cap}\Chow^{\Delta}_{-n}(\mathcal{L}) \cong \Chow^{\Delta}_{-n}(\mathcal{L})$ by Remark \ref{remcapnormcomp} (see also Lemma \ref{lmachoweq}).
\end{ex}

We now assume that the dimension function $\dim$ for $\mathcal{T} = \mathrm{D}(\mathbf{A})$ is given as the opposite of the Krull codimension and furthermore, that the triangulated Gersten conjecture holds for $\mathbf{A}$ and for this choice of dimension function in bidegrees $(i,j)$ with $-p-2 \leq i \leq 0$ and $-1 \leq j \leq p$. Splicing the short exact sequences (\ref{ses}) together yields a partial flasque resolution of the sheaf $\mathscr{K}^{(0)}_{p}$
\begin{equation}
\mathscr{K}^{(0)}_{p} \to \mathscr{K}^{(0/-1)}_{p} \to  \cdots \to \mathscr{K}^{(-p+1/-p)}_{1} \overset{\delta_1}{\longrightarrow} \mathscr{K}^{(-p/-p-1)}_{0} \overset{\delta_0}{\longrightarrow} \mathscr{K}^{(-p-1/-p-2)}_{-1}
\label{res}
\end{equation}
that we can use to calculate its cohomology.

\begin{thm}[Triangulated Bloch formula] 
	Assume that the dimension function $\dim$ for $\mathcal{T}$ is given as the opposite of the Krull codimension and that the triangulated Gersten conjecture holds for $\mathbf{A}$ and for this choice of dimension function in bidegrees $(i,j)$ with $-p-2 \leq i \leq 0$ and $-1 \leq j \leq p$. Then we have isomorphisms
	\[\prescript{}{\cap}\Chow^{\Delta}_{-p}(\mathcal{T}) \cong \mathrm{H}^p(\Spc(\mathcal{T}),\mathscr{K}^{(0)}_p)\]
	for all $p \in \mathbb{Z}$.
	\label{bloch}
\end{thm}
\begin{proof}
	We will use the partial flasque resolution (\ref{res}) of $\mathscr{K}^{(0)}_p$ to calculate the group $\mathrm{H}^p(\Spc(\mathcal{T}),\mathscr{K}^{(0)}_p)$.
	The maps 
	\[\mathscr{K}^{(-p+1/-p)}_{1} \to \mathscr{K}^{(-p/-p-1)}_{0} \to \mathscr{K}^{(-p-1/-p-2)}_{-1}\]
	are spliced together from the exact sequences (\ref{ses}) in the following way:
	
	\[
	\xymatrix{
		&				 & 				    & 0 \ar[d] 				 &  \\
		0 \ar[r] & \mathscr{K}^{(-p+1)}_{1} \ar[r] & \mathscr{K}^{(-p+1/-p)}_{1} \ar[r]^(0.6){\alpha} \ar[rd]_{\delta_1} & \mathscr{K}^{(-p)}_{0} \ar[r] \ar[d]^{\beta} & 0\\
		&				 & 				    & \mathscr{K}^{(-p/-p-1)}_{0} \ar[d]_{\gamma} \ar[rd]^{\delta_0}	 & \\
		&				 & 0 \ar[r]			    & \mathscr{K}^{(-p-1)}_{-1} \ar[d] \ar[r]_(0.4){\epsilon} &  \mathscr{K}^{(-p-1/-p-2)}_{-1}\\
		&				 & 				    & 0					 &  	&\\
	}
	\]
	In order to calculate cohomology, we apply the global section functor.
	As taking global sections is a left-exact functor, $\Gamma(\epsilon)$ is injective and so we have that \[\ker(\Gamma(\delta_0)) = \ker(\Gamma(\gamma)) = \Gamma(\ker(\gamma)) = \Gamma(\im(\beta)) = \prescript{}{\cap}\Cyc^{\Delta}_{-p}(\mathcal{T})~,\]
	again by left-exactness of the global section functor.

	Recall that the maps $\alpha, \beta$ are given as sheafifications of maps $\alpha', \beta'$ between the corresponding presheaves. By the functoriality of sheafification it follows that $\beta \circ \alpha$ is given as the sheafification of the composition $\beta' \circ \alpha'$. But $\beta' \circ \alpha'$ is already a map of sheaves and we therefore have that $\beta \circ \alpha = \beta' \circ \alpha'$. The map $\Gamma(\beta \circ \alpha)$ is therefore given as the composition of the maps
	\[\psi: \mathbb{K}_1\left((\mathbf{A}_X)_{(-p+1)/(-p)}\right) \to \Kzero\left(\mathcal{T}_{(-p)}^{\natural}\right)\]
	with $X= \Spc(\mathcal{T})$ and
	\[\phi: \Kzero\left(\mathcal{T}_{(-p)}^{\natural}\right) \to \Kzero\left(\mathcal{T}_{(-p)}^{\natural}/\mathcal{T}^{\natural}_{(-p-1)}\right)\]
	from the corresponding localization sequences. By the exactness of the localization sequence, $\im(\psi) = \ker(i)$ with 
	\[i: \Kzero\left(\mathcal{T}_{(-p)}^{\natural}\right) \to \Kzero\left(\mathcal{T}_{(-p+1)}^{\natural}\right)\] 
	as in Definition \ref{dfnaltchow2}. Thus, we obtain $\im(\Gamma(\beta \circ \alpha)) = \phi(\ker(i))$.
	
	By our previous calculations we conclude that
	\begin{align*}
	\mathrm{H}^p(\Spc(\mathcal{T}),\mathscr{K}^{(0)}_p) &= \ker(\Gamma(\delta_0))/\im(\Gamma(\delta_1)) \\
	&= \prescript{}{\cap}\Cyc^{\Delta}_{-p}(\mathcal{T})/j \circ \phi(\ker(i)) \\
	&= \prescript{}{\cap}{\Chow}^{\Delta}_{-p}(\mathcal{T})
	\end{align*}
	which was to be shown.
\end{proof}

\begin{rem}
	From the proof of Theorem \ref{bloch}, we can get a simpler definition of $\prescript{}{\cap}{\Cyc}^{\Delta}_{-p}(\mathcal{T})$, not using $\mathrm{K}$-theory sheaves. Namely, we see that the map of sheaves 
	\[\epsilon \circ \gamma: \mathscr{K}^{(-p/-p-1)}_0 \to \mathscr{K}^{(-p-1/-p-2)}_{-1}\] 
	can be computed on global sections as the composition of the maps
	\[\gamma': \Kzero\left(\mathcal{T}_{(-p)}/\mathcal{T}_{(-p-1)} \right)^{\natural} \to \mathbb{K}_{-1}\left((\mathbf{A}_{X})_{(-p-1)}\right)\]
	with $X = \Spc(\mathcal{T})$ and
	\[\epsilon': \mathbb{K}_{-1}\left((\mathbf{A}_{X})_{(-p-1)}\right) \to \mathbb{K}_{-1}\left((\mathbf{A}_X)_{(-p-1)/(-p-2)}\right)~,\]
	both coming from the corresponding long exact localization sequences. We therefore see that 
	\[\prescript{}{\cap}\Cyc^{\Delta}_{-p}(\mathcal{T}) = \Gamma(\im(\beta)) = (\gamma')^{-1}(\ker(\epsilon'))~.\]
	This reformulation of Definition \ref{dfnaltchow2} has the disadvantage that it needs tensor Frobenius pairs in order to talk about $\mathbb{K}_{-1}$ and it is not immediately visible that it is actually independent of a choice of such a tensor Frobenius pair.
\end{rem}

\section{The intersection product}
Recall our assumptions for $\mathcal{T}$ from Convention \ref{convbasicass}. We now let $\dim$ be the opposite of the Krull codimension and require furthermore that the triangulated Gersten conjecture holds for $\mathbf{A}$ in bidegrees $(i,j)$ with $-p-2 \leq i \leq 0$ and $-1 \leq j \leq p$.

First, let us recall a general well-known fact about cup products in sheaf cohomology (see \cite{bredonsheaf}*{Theorem 7.1 and Proposition 7.2}). Let $X$ be a topological space and $\mathcal{F}, \mathcal{G}$ be sheaves of abelian groups on $X$. Then there exists a unique associative bilinear product
\[\cup: \mathrm{H}^p(X,\mathcal{F}) \times \mathrm{H}^q(X,\mathcal{G}) \to \mathrm{H}^{p+q}(X,\mathcal{F} \otimes_{\mathbb{Z}} \mathcal{G}) \]
for all $p,q \in \mathbb{Z}_{\geq 0}$ such that for $p=q=0$, the product is the one induced by the tensor product $\Gamma(X,\mathcal{F}) \times \Gamma(X,\mathcal{G}) \to \Gamma(X,\mathcal{F} \otimes \mathcal{G})$ and the axioms of \cite{bredonsheaf}*{Theorem 7.1} are satisfied. The product $\cup$ is called the \emph{cup product}.

An application of Theorem \ref{bloch} then yields the following:
\begin{cor}
	Under the assumptions of Theorem \ref{bloch} and for $p,q \in \mathbb{Z}_{\geq 0}$ there are bilinear maps
	\[\prescript{}{\cap}\Chow^{\Delta}_{-p}(\mathcal{T}) \times \prescript{}{\cap}\Chow^{\Delta}_{-q}(\mathcal{T}) \to \mathrm{H}^{p+q}\left(\Spc(\mathcal{T}),\mathscr{K}^{(0)}_p \otimes_{\mathbb{Z}} \mathscr{K}^{(0)}_q\right) ~.\]
	\label{intersectionstep1}
\end{cor}

In order to construct the intersection product, we need a map $\mathscr{K}^{(0)}_p \otimes_{\mathbb{Z}} \mathscr{K}^{(0)}_q \to \mathscr{K}^{(0)}_{p+q}$, which will then induce the product map 
\[\mathrm{H}^{p+q}(\Spc(\mathcal{T}),\mathscr{K}^{(0)}_p \otimes_{\mathbb{Z}} \mathscr{K}^{(0)}_q) \to \mathrm{H}^{p+q}(\Spc(\mathcal{T}),\mathscr{K}^{(0)}_{p+q}) = \prescript{}{\cap}\Chow^{\Delta}_{-p -q}(\mathcal{T})\] 
It will be derived from a bilinear map on Waldhausen $\mathrm{K}$-theory induced by the tensor product.

\begin{lma}
	Let $\mathbf{A} = (\mathcal{A},\mathcal{A}_0, \otimes)$ be a tensor Frobenius pair. If we consider $\mathcal{A}$ as a Waldhausen category, then $\otimes$ is a biexact functor $\mathcal{A} \times \mathcal{A} \to \mathcal{A}$ in the sense of \cite{waldktheory}*{Section 1.5}.
	\label{tensorwaldexact}
\end{lma}
\begin{proof}
	By assumption, $a \otimes -$ is an exact functor for all objects $a$ of $\mathcal{A}$ which implies that it preserves cofibrations, as those are just the inflations. Let $f: x \to y$ be a weak equivalence, i.e.\ a map that becomes an isomorphism after passing to $\mathrm{D}(\mathbf{A})$. This means that the object $\mathrm{cone}(f)$ of $\underline{\mathcal{A}}$ is in $\underline{\mathcal{A}_0}$. As $\underline{\mathcal{A}_0}$ is a tensor ideal in $\underline{\mathcal{A}}$ and passing from $\mathcal{A}$ to the stable category $\underline{\mathcal{A}}$ preserves tensor products, it follows that $\mathrm{id}_a \otimes f$ is an isomorphism in $\mathrm{D}(\mathbf{A})$ as well. Therefore $a \otimes - $ preserves weak equivalences. Finally, it is proved in \cite{buehlerexact}*{Proposition 5.2} that exact functors of exact categories preserve pushouts along inflations, which in our case means that pushouts along weak equivalences are preserved. Therefore, the functors $a \otimes -$ (and by symmetry $- \otimes a$) are exact in the sense of Waldhausen (see \cite{waldktheory}*{Section 1.5}).
	
	It remains to check that $\otimes$ satisfies the ``more technical condition'' of \cite{waldktheory}*{Section 1.5}. This asserts that for two cofibrations $\alpha: a \rightarrowtail a', \beta:b \rightarrowtail b'$ in the diagram
	\[
	\xymatrix{
		a \otimes b \ar@{>->}[r]^{\alpha \otimes \mathrm{id}_b} \ar@{>->}[d]_{\mathrm{id}_a \otimes \beta} & a' \otimes b \ar[d] \ar@/^1pc/[rdd]^{\mathrm{id}_{a'} \otimes \beta}& \\
		a \otimes b' \ar[r] \ar@/_1pc/[rrd]_{\alpha \otimes \mathrm{id}_{b'}}& a' \otimes b \coprod\limits_{a \otimes b} a \otimes b'  \ar@{.>}[rd]_{\phi}& \\
		& & a' \otimes b'
	}
	\]
	the arrow $\phi$ is a cofibration, i.e.\ an inflation. This is exactly axiom \ref{listtensorfrobdefpushprod} of Definition~\ref{tensorfrobdef}.
\end{proof}

A biexact functor $\otimes: \mathcal{A} \times \mathcal{A} \to \mathcal{A}$ in the above sense gives rise to bilinear maps
\[\mathbb{K}_p(\mathbf{A}) \otimes_{\mathbb{Z}}  \mathbb{K}_q(\mathbf{A}) \to \mathbb{K}_{p+q}(\mathbf{A})\]
for all $p,q \geq 0$ (see \cite{waldktheory}*{Section 1.5}). In particular, we obtain maps
\[\mathbb{K}_p(\mathbf{A}_U) \otimes_{\mathbb{Z}}  \mathbb{K}_q(\mathbf{A}_U) \to \mathbb{K}_{p+q}(\mathbf{A}_U)\]
as $\mathbf{A}_U:= (\mathbf{A}_U)_{(0)}$ inherits the structure of a tensor Frobenius pair from $\mathbf{A}$ by Corollary~\ref{corsubtfp} and Lemma~\ref{lmaictfp}. These maps sheafify to 
\[\mathscr{K}^{(0)}_p \otimes_{\mathbb{Z}} \mathscr{K}^{(0)}_q \to  \mathscr{K}^{(0)}_{p+q}.\]
and give us
\begin{equation}
\mathrm{H}^{p+q}\left(\Spc(\mathcal{T}), \mathscr{K}^{(0)}_p \otimes_{\mathbb{Z}} \mathscr{K}^{(0)}_q\right) \to \mathrm{H}^{p+q}\left(\Spc(\mathcal{T}), \mathscr{K}^{(0)}_{p+q}\right) = \prescript{}{\cap}\Chow^{\Delta}_{-p -q}(\mathcal{T})
\label{intersectionstep2}
\end{equation}
for all $p,q \geq 0$.

\begin{dfn} 
	Let $\mathbf{A}$ be a tensor Frobenius pair with derived category $\mathcal{T}$ that satisfies the assumptions of Theorem \ref{bloch}. For $p,q \in \mathbb{Z}_{\geq 0}$, the \emph{intersection product} is the bilinear map
	\[\alpha: \prescript{}{\cap}\Chow^{\Delta}_{-p}(\mathcal{T}) \otimes \prescript{}{\cap}\Chow^{\Delta}_{-q}(\mathcal{T}) \to \prescript{}{\cap}\Chow^{\Delta}_{-p -q}(\mathcal{T})\]
	that arises as the composition of the map in Corollary~\ref{intersectionstep1} and in (\ref{intersectionstep2}).
	\label{productdef}
\end{dfn}

\begin{rem}
	While the groups $\prescript{}{\cap}\Chow^{\Delta}_{n}(\mathcal{T})$ only depend on $\mathrm{D}(\mathbf{A}) = \mathcal{T}$, the product $\alpha$ of Definition \ref{productdef} might depend on the whole model $\mathbf{A}$.
\end{rem}

\begin{rem}
	Let $\mathbf{A},\mathbf{B}$ be two tensor Frobenius pairs satsifying the assumptions of Theorem \ref{bloch} and $F: \mathbf{A} \to \mathbf{B}$ a map of tensor Frobenius pairs (i.e.\ a map of Frobenius pairs that respects the tensor products up to natural isomorphism) such that the induced maps on the derived categories has relative dimension~0 (see \cite{kleinchow}*{Definition 4.1.1}), i.e.\ $\mathrm{D}(F)(\mathrm{D}(\mathbf{A})_{(p)}) \subset \mathrm{D}(\mathbf{B})_{(p)}$ for all $p \in \mathbb{Z}_{\geq 0}$. Then $F$ induces maps 
	\[\prescript{}{\cap}\Chow(F)_{-p}: \prescript{}{\cap}\Chow^{\Delta}_{-p}\left(\mathrm{D}(\mathbf{A})\right) \to \prescript{}{\cap}\Chow^{\Delta}_{-p}\left(\mathrm{D}(\mathbf{B})\right)\]
	for all $p \in \mathbb{Z}_{\geq 0}$ and there is a commutative diagram
	\[
	\xymatrix{
		\prescript{}{\cap}\Chow^{\Delta}_{-p}\left(\mathrm{D}(\mathbf{A})\right) \times \prescript{}{\cap}\Chow^{\Delta}_{-q}\left(\mathrm{D}(\mathbf{A})\right) \ar[r]^(0.6){\alpha_{\mathbf{A}}} \ar[d]_{\prescript{}{\cap}\Chow(F)_{-p} \times \prescript{}{\cap}\Chow(F)_{-q}} & \prescript{}{\cap}\Chow^{\Delta}_{-p-q}\left(\mathrm{D}(\mathbf{A})\right) \ar[d]^{\prescript{}{\cap}\Chow(F)_{-p-q}}\\
		\prescript{}{\cap}\Chow^{\Delta}_{-p}\left(\mathrm{D}(\mathbf{B})\right) \times \prescript{}{\cap}\Chow^{\Delta}_{-q}\left(\mathrm{D}(\mathbf{B})\right) \ar[r]^(0.6){\alpha_{\mathbf{B}}} &  \prescript{}{\cap}\Chow^{\Delta}_{-p-q}\left(\mathrm{D}(\mathbf{B})\right)
	}
	\]
	with $\alpha_{\mathbf{A}}, \alpha_{\mathbf{B}}$ the respective products from Definition \ref{productdef}. In this sense, the construction is functorial.
\end{rem}

\begin{rem}
	The author expects that properties of the product in Waldhausen $\mathrm{K}$-theory and the cup product in sheaf cohomology as in Corollary \ref{intersectionstep1} imply that the intersection product from Definition \ref{productdef} makes
	\[\bigoplus_{p \geq 0} \prescript{}{\cap}\Chow^{\Delta}_{-p}(\mathcal{T})\]
	a graded-commutative ring with unit the class of $\mathbb{I}$ in $\prescript{}{\cap}\Chow^{\Delta}_{0}(\mathcal{T})$.
	
	If $G = \mathbb{Z}/2\mathbb{Z} \times \mathbb{Z}/2\mathbb{Z}$ and $k$ is an algebraically closed field of characteristic $2$, the results in \cite{kleinchow}*{Section 6.4} show that
	\[\bigoplus_{p \geq 0} \prescript{}{\cap}\Chow^{\Delta}_{-p}(kG\stab) = \mathbb{Z}/2\mathbb{Z} \oplus \mathbb{Z}/2\mathbb{Z}~.\]
	The only possible commutative unital ring structure on this group, that also has a nilpotent element is $(\mathbb{Z}/2\mathbb{Z})[\epsilon]/(\epsilon^2)$. Thus, if the above assumption holds true, any choice of tensor Frobenius pair with derived category $kG\stab$ that satisfies the triangulated Gersten conjecture in the relevant degrees (if it exists) must yield the same intersection product.
\end{rem}

\section{Example: strict perfect complexes on a non-singular algebraic variety}
\label{theintproduct}
We now introduce the main example of Definition \ref{productdef} which justifies the name ``intersection product''. Let $X$ be a non-singular separated scheme of finite type over a field. Recall that a \emph{strict perfect complex} on $X$ is a bounded complex of locally free $\mathcal{O}_X$-modules of finite rank.
\begin{dfn}
	Let $\mathrm{sPerf}$ denote the category of strict perfect complexes on $X$ endowed with the following structure of exact category: a sequence of strict perfect complexes
	\[\mathcal{F}^{\bullet} \to \mathcal{G}^{\bullet} \to \mathcal{H}^{\bullet}\]
	is a conflation if it is degree-wise a split exact sequence. We denote the full subcategory of acyclic strict perfect complexes by $\mathrm{asPerf}$.
	\label{dfnperf}
\end{dfn}

\begin{lma}
	The triple $\mathbf{sPerf}=(\mathrm{sPerf},\mathrm{asPerf},\otimes_{\mathcal{O}_X})$ is a tensor Frobenius pair.
	\label{lmasperfpushprod}
\end{lma}
\begin{proof}
	For an exact category $\mathcal{E}$, let $\mathrm{Ch}^b(\mathcal{E})$ denote the exact category of all bounded chain complexes over $\mathcal{E}$, with the conflations defined as the degree-wise split exact sequences. Let $\mathrm{Ac}^b(\mathcal{E}) \subset \mathrm{Ch}^b(\mathcal{E})$ denote full subcategory of acyclic complexes. In \cite{schlichtingnegative}*{Section 5.3}, it is shown that $(\mathrm{Ch}^b(\mathcal{E}), \mathrm{Ac}^b(\mathcal{E}))$ is a Frobenius pair. Thus, when we consider the full subcategory of locally free sheaves of finite rank in $\mathrm{Coh}(X)$ as an exact category, it follows that $(\mathrm{sPerf},\mathrm{asPerf})$ is a Frobenius pair. 
	
	It is clear that the tensor product of two strict perfect complexes is again strict perfect and as tensoring with a strict perfect complex is an exact functor, it follows that $\mathrm{asPerf}$ is a $\otimes_{\mathcal{O}_X}$-ideal. It remains to check that the pushout product axiom of Definition \ref{dfntensorfrobpair} holds true. Thus, let $f: A_{\bullet} \rightarrowtail X_{\bullet}$ and $g: B_{\bullet} \rightarrowtail Y_{\bullet}$ be two inflations in $\mathrm{sPerf}$. This means that for each $i \in \mathbb{Z}$ we have automorphisms $\alpha_i: X_i \to X_i$ and $\beta_i: Y_i \to Y_i$ such that $\alpha_i \circ f_i$ is a split injection $A_i \hookrightarrow A_i \oplus C_i$ and $\beta_i \circ g_i$ is a split injection $B_i \hookrightarrow B_i \oplus D_i$. The maps $f \otimes \id_{B_{\bullet}}$ and $\id_{A_{\bullet}} \otimes g$ are given componentwise as
	\begin{align*}
	(f \otimes \id_{B_{\bullet}})_{k}: \bigoplus_{i+j = k} A_i \otimes B_j & \to \bigoplus_{i+j = k} X_i \otimes B_j\\
	(\id_{A_{\bullet}} \otimes g)_{k}: \bigoplus_{i+j = k} A_i \otimes B_j & \to \bigoplus_{i+j = k} A_i \otimes Y_j
	\end{align*}
	and after post-composing with the isomorphisms consisting of diagonal matrices with entries $\alpha_i \otimes \id_{B_j}$ and $\id_{A_i} \otimes \beta_j$ respectively, we obtain split injections
	\begin{align*}
	\bigoplus_{i+j = k} A_i \otimes B_j & \to \bigoplus_{i+j = k} (A_i \otimes B_j) \oplus (C_i \otimes B_j)\\
	\bigoplus_{i+j = k} A_i \otimes B_j & \to \bigoplus_{i+j = k} (A_i \otimes B_j) \oplus (A_i \otimes D_j)~.
	\end{align*}
	We see that therefore 
	\[\left((A_{\bullet} \otimes Y_{\bullet}) \coprod_{A_{\bullet} \otimes B_{\bullet}} (X_{\bullet} \otimes B_{\bullet})\right)_k \cong \bigoplus_{i+j = k} (A_i \otimes B_j) \oplus (A_i \otimes D_j) \oplus (C_i \otimes B_j)\]
	Similarly, we see that 
	\[(X_{\bullet} \otimes Y_{\bullet})_k \cong \bigoplus_{i+j = k} (A_i \otimes B_j) \oplus (A_i \otimes D_j) \oplus (C_i \otimes B_j) \oplus (C_i \otimes D_j)\]
	and the induced map 
	\[\left((A_{\bullet} \otimes Y_{\bullet}) \coprod_{A_{\bullet} \otimes B_{\bullet}} (X_{\bullet} \otimes B_{\bullet})\right)_k \longrightarrow (X_{\bullet} \otimes Y_{\bullet})_k\]
	is given as the canonical inclusion, which is split. This shows that the pushout product axiom holds in $\mathrm{sPerf}$ and finishes the proof of the lemma.
\end{proof}

\begin{lma}
	The category $\mathrm{D}(\mathbf{sPerf})$ is equivalent to $\Dperf(X)$ as a tensor triangulated category.
	\label{lmastrictlrelax}
\end{lma}
\begin{proof}
	The inclusion functor from the exact category of strict perfect complexes into the exact category of perfect complexes induces an exact equivalence of derived categories between $\mathrm{D}(\mathbf{sPerf})$ and $\Dperf(X)$ if $X$ has an ample family of line bundles, as follows from \cite{thomason-trobaugh}*{Proposition 2.3.1}, as mentioned in the proof of \cite{thomason-trobaugh}*{Lemma 3.8}. As being noetherian, separated and regular already implies that $X$ admits an ample family of line bundles (see \cite{sga6}*{Corollaire 2.2.7.1}), our assumptions on $X$ guarantee that the inclusion is an equivalence. It is also a tensor functor as we can compute the derived tensor product by tensoring with a quasi-isomorphic strict perfect complex.
\end{proof}

\begin{conv}
	For the remaining part of the section, we set $\mathbf{T} := \mathbf{sPerf}$ and $\mathcal{T} := \mathrm{D}(\mathbf{sPerf}) \cong \Dperf(X)$. We fix the opposite of the codimension of support as a dimension function on $\mathcal{T}$.
\end{conv}

\begin{lma}
	The Frobenius pair $\mathbf{sPerf}$ satisfies the triangulated Gersten conjecture in bidegrees $(l,p)$ for $l \leq 0$ and $p \geq -1$.
	\label{gerstenexample}
\end{lma}
\begin{proof}
	First, let us introduce some maps of exact sequences of Frobenius pairs, which will allow us to ged rid of idempotent completions and work with complexes of coherent sheaves instead of perfect ones. 
	
	For $U \subset X$ open with complement $Z$, we start with
	\begin{equation}
	\begin{gathered}
	\xymatrix{
		(\mathbf{T}_U)_{(l-1)} \ar[r] &  (\mathbf{T}_U)_{(l)} \ar[r] & (\mathbf{T}_U)_{(l)/(l-1)} \\
		\left(\mathbf{sPerf}^{~U}_{(l-1)},\mathbf{sPerf}_{Z}\right) \ar[r] \ar[u] & \left(\mathbf{sPerf}^{~U}_{(l)},\mathbf{sPerf}_{Z}\right) \ar[r] \ar[u] & \left(\mathbf{sPerf}^{~U}_{(l)},\mathbf{sPerf}^{~U}_{(l-1)}\right) \ar[u]
	}
	\label{eqexseq1}
	\end{gathered}
	\end{equation}
	in the notation of Definition \ref{ksheafdef}, where the vertical arrows are given as the inclusion into the countable envelope. Since $X$ is regular, $\mathcal{T}/\mathcal{T}_Z$ is equivalent to $\Db(\mathrm{Coh}(U))$ (see e.g.\ \cite{kleinchow}*{Section 3.2}), which is already idempotent complete. Hence, the vertical arrows induce equivalences of the corresponding derived categories since $X$ is regular. Therefore they induce isomorphisms in $\mathbb{K}$-theory.
	
	For an abelian category $\mathcal{A}$, define the Frobenius pair 
	\[\mathbf{Ch}^b(\mathcal{A}):=(\mathrm{Ch}^b(\mathcal{A}), \mathrm{Ac}^b(\mathcal{A}))~,\] 
	where $\mathrm{Ch}^b(\mathcal{A})$ is the category of bounded chain complexes in $\mathcal{A}$ and $\mathrm{Ac}^b(\mathcal{A})$ is the full subcategory of complexes homotopy equivalent to an acyclic chain complex. The conflations in $\mathbf{Ch}^b(\mathcal{A})$ are by definition the degree-wise split exact sequences. There is a  map of exact sequences of Frobenius pairs
	\begin{equation}
	\begin{gathered}
	\xymatrix{
		\left(\mathbf{sPerf}^{~U}_{(l-1)},\mathbf{sPerf}_{Z}\right)  \ar[d] \ar[r]& \left(\mathbf{sPerf}^{~U}_{(l)},\mathbf{sPerf}_{Z}\right)  \ar[d] \ar[r]& \left(\mathbf{sPerf}^{~U}_{(l)},\mathbf{sPerf}^{~U}_{(l-1)}\right) \ar[d]\\
		\mathbf{Ch}^b\left((\mathrm{Coh}(U)_{(l-1)}\right) \ar[r]  & \mathbf{Ch}^b\left(\mathrm{Coh}(U)_{(l)}\right) \ar[r]  & \mathbf{Ch}^b\left(\mathrm{Coh}(U)_{(l)}/\mathrm{Coh}(U)_{(l-1)}\right)
	}
	\label{eqexseq2}
	\end{gathered}
	\end{equation}
	where the vertical maps are given by restriction to $U$. Again, we check that they induce equivalences of the corresponding derived categories and therefore induce isomorphisms in $\mathbb{K}$-theory. 
	
	Using the maps (\ref{eqexseq1}) and (\ref{eqexseq2}) and \cite{schlichtingnegative}*{Theorem 11.10}, we see that the localization sequences corresponding to
	\[(\mathbf{T}_U)_{(l-1)} \to  (\mathbf{T}_U)_{(l)} \to (\mathbf{T}_U)_{(l)/(l-1)}\]
	and
	\[\mathbf{Ch}^b\left((\mathrm{Coh}(U)_{(l-1)}\right) \to \mathbf{Ch}^b\left(\mathrm{Coh}(U)_{(l)}\right) \to \mathbf{Ch}^b\left(\mathrm{Coh}(U)_{(l)}/\mathrm{Coh}(U)_{(l-1)}\right)\]
	are isomorphic. This proves the lemma for $p=-1$ by \cite{schlichtingnegative}*{Theorem 9.1}, which shows that $\mathbb{K}_{-1}(\mathbf{Ch}^b(\mathcal{A})) = 0$ for any abelian category $\mathcal{A}$. For $p \geq 0$, \cite{thomason-trobaugh}*{Theorem 1.11.7} shows that both localization sequences are in turn isomorphic to the localization sequence
	\[\cdots \to \mathrm{K}_{p}\left(\mathrm{Coh}(U)_{(l)}\right) \to \mathrm{K}_{p}\left(\frac{\mathrm{Coh}(U)_{(l)}}{\mathrm{Coh}(U)_{(l-1)}}\right) \to \mathrm{K}_{p-1}\left(\mathrm{Coh}(U)_{(l-1)}\right) \to \cdots \]
	from Quillen $\mathrm{K}$-theory for all $l \in \mathbb{Z}$, where $\mathrm{Coh}(U)_{(l)}$ denotes the abelian category of coherent sheaves on the open subscheme $U \subset X$, with codimension of support $\geq -l$.
	
	Therefore the stalks of the exact sequence (\ref{longexactgersten}) are exact sequences isomorphic to the usual ones in the Gersten conjecture, which is satisfied for regular local rings of finite type over a field (see \cite{quillenhigher}*{Theorem 5.11}). This implies the statement as we can check the vanishing of a map of sheaves on the stalks.
\end{proof}

\begin{lma}
	There are isomorphisms
	\[\prescript{}{\cap}\Chow^{\Delta}_{-p}(\mathcal{T}) \cong \Chow^{p}(X)\]
	for all $p \in \mathbb{Z}$.
	\label{lmachoweq}
\end{lma}
\begin{proof}
	Under our assumptions, Theorem \ref{agreementthm} shows that $\Chow^{\Delta}_{-p}(\mathcal{T}) \cong \Chow^{p}(X)$ for all $p \in \mathbb{Z}$. The isomorphisms $\prescript{}{\cap}\Chow^{\Delta}_{-p}(\mathcal{T}) \cong \Chow^{\Delta}_{-p}(\mathcal{T})$ are a consequence of the fact that the categories $\mathcal{T}_{(-p)}^{\natural}/\mathcal{T}_{(-p-1)}^{\natural}$ can be expressed as bounded derived categories of abelian categories since we assumed that $X$ is non-singular (namely $\mathcal{T}_{(-p)}^{\natural}/\mathcal{T}_{(-p-1)}^{\natural} \cong \Db(\mathrm{Coh}(X))_{(-p)}/\Db(\mathrm{Coh}(X))_{(-p-1)} \cong \Db(\mathrm{Coh}(X)_{(-p)}/\mathrm{Coh}(X)_{(-p-1)} )$), and are therefore idempotent complete already. Thus, there is an equivalence 
	\[\mathcal{T}_{(-p)}^{\natural}/\mathcal{T}_{(-p-1)}^{\natural} \to \left(\mathcal{T}_{(-p)}^{\natural}/\mathcal{T}_{(-p-1)}^{\natural}\right)^{\natural} \cong \left(\mathcal{T}_{(-p)}/\mathcal{T}_{(-p-1)}\right)^{\natural}\]
	induced by the inclusion functor, which gives the isomorphism by Remark \ref{remcapnormcomp}.
\end{proof}

We now want to compare the usual intersection product on $X$ and the product from Definition \ref{productdef} on the tensor triangular Chow groups of $\Dperf(X)$, coming from the tensor Frobenius pair $\mathbf{sPerf}$. In order to do this, consider the isomorphisms 
\[\mathbb{K}_i(\mathbf{T}_U) \to \mathrm{K}_i(\mathrm{Coh}(U))\] 
that were constructed in the proof of Lemma \ref{gerstenexample}. If we denote them by $s^U_i$, then for all $i,j \geq 0$ and $U \subset X$ open, they fit into a diagram
\begin{equation}
\begin{aligned}
\xymatrix{
	\mathbb{K}_i(\mathbf{T}_U) \otimes \mathbb{K}_j(\mathbf{T}_U) \ar[r] \ar[d]^{s^U_i \otimes s^U_j}&  \mathbb{K}_{i+j}(\mathbf{T}_U) \ar[d]^{s^U_{i+j}} \\
	\mathrm{K}_i(\mathrm{Coh}(U)) \otimes \mathrm{K}_j(\mathrm{Coh}(U)) \ar[r]&  \mathrm{K}_{i+j}(\mathrm{Coh}(U))
}
\end{aligned}
\label{conjkprodsame}
\end{equation}
where the horizontal arrows are given by the products in the Waldhausen $\mathrm{K}$-theory of $\mathbf{T}_U$ and in the Quillen $\mathrm{K}$-theory of $\mathrm{Coh}(U)$ (see \cite{waldgenprods}), respectively.

\begin{thm} 
	Let $\alpha$ denote the intersection product from Definition~\ref{productdef} coming from the tensor Frobenius pair $\mathbf{sPerf}$ and let $\alpha'$ be the usual intersection product on~$X$. Assume that diagram~(\ref{conjkprodsame}) commutes for all $i,j \geq 0$ and all opens $U \subset X$. Then the diagram
	\[
	\xymatrix{
		\prescript{}{\cap}\Chow^{\Delta}_{-p}(\mathcal{T}) \otimes \prescript{}{\cap}\Chow^{\Delta}_{-q}(\mathcal{T}) \ar[r]^(0.6){\alpha} \ar[d]^{\cong}& \prescript{}{\cap}\Chow^{\Delta}_{-p-q}(\mathcal{T}) \ar[d]^{\cong} \\
		\Chow^{p}(X) \otimes \Chow^{q}(X) \ar[r]^(0.6){\alpha'} & \Chow^{p+q}(X)
	}
	\]
	commutes up to a sign $(-1)^{pq}$ for all $p,q \geq 0$.
	\label{thmprodagrees}
\end{thm}

\begin{rem}
	The construction of the products in Quillen and Waldhausen $\mathrm{K}$-theory is so natural that it seems very plausible that diagram~(\ref{conjkprodsame}) always commutes for all $i,j \geq 0$ and all opens $U \subset X$. A hypothetical proof of this statement involves comparing the products in Waldhausen and Quillen $\mathrm{K}$-theory via the homotopy equivalence between the Waldhausen and Quillen $\mathrm{K}$-theory spectra of an exact category constructed in \cite{waldktheory}*{Section 1.9}. The author plans to further investigate this in future work.
\end{rem}

\begin{proof}[Proof of Theorem \ref{thmprodagrees}]
	As we have $\Spc(\mathcal{T}) \cong X$ and $X$ is regular, the sheaves $\mathscr{K}^{(0)}_p$ will be isomorphic to the sheaves $\mathcal{F}_p$ associated to the presheaf $U \mapsto \mathrm{K}_p(\mathrm{Coh}(U))$ on $X$ via the isomorphisms $s_p$. By the Bloch formula, $\mathrm{H}^{p}(X,\mathcal{F}_p) \cong \Chow^{p}(X)$. The statement now follows by the commutativity of diagram (\ref{conjkprodsame}) and the main result of \cite{graysonproduct}, where it is shown that the product
	\[\mathrm{H}^{p}(X,\mathcal{F}_p) \otimes \mathrm{H}^{q}(X,\mathcal{F}_q) \to H^{p+q}(X,\mathcal{F}_p \otimes \mathcal{F}_q) \to \mathrm{H}^{p+q}(X,\mathcal{F}_{p+q})\]
	agrees with the usual intersection product up to a sign $(-1)^{pq}$, where the second map comes from the product on Quillen $\mathrm{K}$-theory induced by the tensor product.
\end{proof}

\newpage
\begin{appendices}
\section{The countable envelope of a tensor Frobenius pair}
\label{chaptercountable}
In this section, we show that the countable envelope of a tensor Frobenius pair (see Definition \ref{dfntensorfrobpair}) naturally inherits the structure of a tensor Frobenius pair. This is an extension of work of Keller \cite{kellerchain}*{Appendix B} and Schlichting \cite{schlichtingnegative}*{Section 4} to a symmetric monoidal setting. It enables us to embed a tensor Frobenius pair into one that has countable coproducts, which in turn makes it possible to lift the embedding of its derived category into its idempotent completion to the level of Frobenius models (see Section \ref{sectmodic}).

\subsection{Ind-objects in an additive category}
In this section we recall some of the theory of ind-objects in the additive setting. We heavily rely on the exposition in \cite{kaschap}.

Let $\mathcal{E}$ be a small additive category and denote by $\hat{\mathcal{E}}^{\mathrm{add}} := \mathrm{Funct}_{\mathrm{add}}(\mathcal{E}^{\mathrm{op}},\mathbf{Ab})$ the abelian category of additive functors from $\mathcal{E}$ to the category of Abelian groups. By composition with the forgetful functor, it can be considered as a full subcategory of the category of all functors $\hat{\mathcal{E}}:=\mathrm{Funct}(\mathcal{E}^{\mathrm{op}},\mathbf{Set})$ from $\mathcal{E}$ to the category of sets (see \cite{kaschap}*{Proposition 8.2.12}).

The Yoneda functor gives an a priori embedding $\mathcal{E} \to \hat{\mathcal{E}}$, but as Hom-sets are abelian groups and Hom-functors are additive in our setting, it factors through an embedding $h_{\mathcal{E}}: \mathcal{E} \to \hat{\mathcal{E}}^{\mathrm{add}}$. Given a small filtered category $I$ and a functor $\alpha: I \to \mathcal{E}$ in $\mathcal{E}$, its colimit in $\mathcal{E}$ might not exist. We denote by ``${\varinjlim}$'' $\alpha$ the colimit of the inductive system $h_{\mathcal{E}} \circ F$ in $\hat{\mathcal{E}}$, which is also in $\hat{\mathcal{E}}^{\mathrm{add}}$.

\begin{dfn}[cf.\ \cite{kaschap}*{Definition 6.1.1}] 
	An \emph{ind-object} in $\mathcal{E}$ is by definition an object of $\hat{\mathcal{E}}$ that is isomorphic in $\hat{\mathcal{E}}$ to ``${\varinjlim}$'' $\alpha$ for some small filtered category $I$ and a functor~$\alpha: I \to \mathcal{E}$. We denote by $\mathrm{Ind}(\mathcal{E})$ the full subcategory of $\hat{\mathcal{E}}$ consisting of the ind-objects in $\mathcal{E}$. The functor $h_{\mathcal{E}}$ induces a full embedding $\iota_\mathcal{E}: \mathcal{E} \to \mathrm{Ind}(\mathcal{E})$.
	\label{dfnindobdef}
\end{dfn}

\begin{rem}
	In the literature, the category of ind-objects in $\mathcal{E}$ is often defined as the full subcategory of $\hat{\mathcal{E}}$ consisting of filtered colimits of representable functors (see e.g.\ \cite{SGA4-1}). The resulting category $\mathrm{Ind}'(\mathcal{E})$ is equivalent to $\mathrm{Ind}(\mathcal{E})$ from Definition \ref{dfnindobdef} and it is also possible to construct an explicit quasi-inverse to the inclusion $\mathrm{Ind}'(\mathcal{E}) \hookrightarrow \mathrm{Ind}(\mathcal{E})$ as follows: for any object $A \in \mathrm{Ind}(\mathcal{E})$, denote by $\mathcal{E}_A$ the category with objects arrows $s_U: U \to A$ in $\mathrm{Ind}(\mathcal{E})$ with $U \in \mathcal{E}$ (we identify $\mathcal{E}$ with a subcategory of $\mathrm{Ind}(\mathcal{E})$ via $\iota_{\mathcal{E}}$). A morphism $f: s_U \to s_V$ in $\mathcal{E}_A$ is a morphism in $\mathcal{E}$ that makes the diagram in~$\mathrm{Ind}(\mathcal{E})$
	\[
	\xymatrix{
		U \ar[r]^{s_U} \ar[d]_{f} & A \\
		V \ar[ur]_{s_V}}
	\] 
	commute. The category $\mathcal{E}_A$ is cofinally small and filtered by \cite{kaschap}*{Proposition 6.1.5} and thus we can define a functor
	\begin{equation}
	\begin{split}
	\mathrm{Ind}(\mathcal{E}) &\to \mathrm{Ind}'(\mathcal{E}) \subset \mathrm{Ind}(\mathcal{E})\\
	A &\mapsto \underset{(U \to A) \in \mathcal{E}_A}{\text{``$\varinjlim$''}} U
	\end{split}
	\label{eqquinv}
	\end{equation}
	which has image in $\mathrm{Ind}'(\mathcal{E})$. By \cite{kaschap}*{Proposition 2.6.3 (i)}, the natural map 
	\[\underset{(U \to A) \in \mathcal{E}_A}{\text{``$\varinjlim$''}} U \to A\]
	is an isomorphism. If $A = \text{``$\varinjlim$''} \alpha$ for some functor $\alpha: I \to \mathcal{E}$, then there is an associated functor $I \to \mathcal{E}_A$ which maps $i \in I$ to the canonical morphism $\alpha(i) \to A$. This functor is cofinal by \cite{kaschap}*{Proposition 2.6.3 (ii)} and we see that the functor (\ref{eqquinv}) is indeed the desired quasi-inverse.
\end{rem}

Under our assumptions, $\mathrm{Ind}(\mathcal{E})$ carries the expected additional structure.
\begin{lma}
	The category $\mathrm{Ind}(\mathcal{E})$ is additive.
\end{lma}
\begin{proof}
	It is immediate from the definition of $\mathrm{Ind}(\mathcal{E})$ as a full subcategory of $\hat{\mathcal{E}}$ that the category $\mathrm{Ind}(\mathcal{E})$ is pre-additive, i.e.\ the morphism sets are abelian groups and composition is bilinear. As $\mathcal{E}$ is additive it has finite coproducts and by \cite{kaschap}*{Proposition 6.1.18}, it follows that $\mathrm{Ind}(\mathcal{E})$ admits small (and in particular finite) coproducts. As finite coproducts and products coincide in a pre-additive category (see \cite{kaschap}*{Corollary 8.2.4}), it follows by \cite{kaschap}*{Lemma 8.2.9} that $\mathrm{Ind}(\mathcal{E})$ is additive.
\end{proof}

We finish the section with two statements about the indization of symmetric monoidal categories.
\begin{prop}
	Let $\mathcal{E}$ be endowed with a symmetric monoidal structure such that the functor $a \otimes -$ is additive for all objects $a \in \mathcal{E}$. Then $\mathrm{Ind}(\mathcal{E})$ naturally inherits a symmetric monoidal structure such that the inclusion $h_{\mathcal{E}}:\mathcal{E} \to \mathrm{Ind}(\mathcal{E})$ preserves the tensor product.
	\label{propindmonoidal}
\end{prop}
\begin{proof}
	The statement seems to be well-known for $\mathrm{Ind}'(\mathcal{E})$, at least in the context of abelian monoidal categories (see e.g.\ \cite{delignecattak}*{Section 7} or \cite{haiemb}*{Section 3.4}), where one sets 
	\[\underset{I}{\text{``${\varinjlim}$''}} \alpha \otimes_{\mathrm{I}} \underset{J}{\text{``${\varinjlim}$''}} \beta := \underset{I \times J}{\text{``${\varinjlim}$''}} \alpha \otimes \beta \]
	with $\alpha: I \to \mathcal{E}$ and $\beta: J \to \mathcal{E}$ functors from small filtered categories $I,J$ to $\mathcal{E}$ and $\otimes$ the tensor product on $\mathcal{E}$. Thus, we can define a symmetric monoidal structure on $\mathrm{Ind}(\mathcal{E})$ by pulling back along the equivalence (\ref{eqquinv}). Explicitly, we set for two objects $A,B \in \mathrm{Ind}(\mathcal{E})$ 
	\[A \otimes_{\mathrm{I}} B :=  \underset{((U \to A),(V \to B)) \in \mathcal{E}_A \times \mathcal{E}_B}{\text{``${\varinjlim}$''}} U \otimes V ~.\]
	The unit object of $\mathrm{Ind}(\mathcal{E})$ is given as the image of the unit object of $\mathcal{E}$ under $h_{\mathcal{E}}$ and the associativity, commutativity and unit isomorphisms are all induced by the ones of~$\mathcal{E}$.
\end{proof}

\begin{rem}
	The product $\otimes_I$ is naturally isomorphic to the restriction to $\mathrm{Ind}(\mathcal{E})$  of the \emph{Day convolution product} on $\hat{\mathcal{E}}$ (see \cite{dayclosedcats}). This product commutes with colimits in both arguments and the Yoneda embedding takes the tensor product on $\mathcal{E}$ to the convolution product on $\hat{\mathcal{E}}$. Therefore, it must be isomorphic to $\otimes_I$.
\end{rem}

\begin{lma}
	In the situation of Proposition \ref{propindmonoidal}, the functor $A \otimes_{\mathrm{I}} -$ is additive for all objects $A \in \mathrm{Ind}(\mathcal{E})$.
\end{lma}
\begin{proof}
	By \cite{kaschap}*{Proposition 8.2.15}, in order to prove additivity, it suffices to show that $A \otimes_{\mathrm{I}} -$ preserves binary products. Assume we are given functors $\alpha: I \to \mathcal{E}, \beta: J \to \mathcal{E}, \gamma: K \to \mathcal{E}$ from small filtered categories $I,J,K$ to $\mathcal{E}$ such that $\underset{I}{\text{``${\varinjlim}$''}} \alpha \cong A, \underset{J}{\text{``${\varinjlim}$''}} \beta \cong B,  \underset{K}{\text{``${\varinjlim}$''}} \gamma \cong C$. Then
	\begin{align*}
	A \otimes_{\mathrm{I}} (B \times C) & \cong \left(\underset{I}{\text{``${\varinjlim}$''}} \alpha \right)  \otimes_{\mathrm{I}} \left(\left(\underset{J}{\text{``${\varinjlim}$''}} \beta \right) \times  \left( \underset{K}{\text{``${\varinjlim}$''}} \gamma \right)\right) \\
	&\cong \left(\underset{I}{\text{``${\varinjlim}$''}} \alpha \right)  \otimes_{\mathrm{I}} \left(\underset{J \times K}{\text{``${\varinjlim}$''}} \beta \times \gamma \right)\\
	&\cong \underset{I \times J \times K}{\text{``${\varinjlim}$''}}\alpha \otimes(\beta \times \gamma)\\
	&\cong \underset{I \times J \times K}{\text{``${\varinjlim}$''}}\alpha \otimes \beta \times \alpha \otimes \gamma ~,\\
	\end{align*}
	where we used that $\text{``${\varinjlim}$''}$ commutes with finite products and that $\otimes$ is additive in each variable. As the diagonal functor $I \to I \times I$ is cofinal (see \cite{kaschap}*{Corollary 3.2.3}), we obtain 
	\begin{align*}
	\underset{I \times J \times K}{\text{``${\varinjlim}$''}}\alpha \otimes \beta \times \alpha \otimes \gamma &\cong \underset{I \times J \times I \times K}{\text{``${\varinjlim}$''}}\alpha \otimes \beta \times \alpha \otimes \gamma \\
	&\cong \left(\underset{I \times J}{\text{``${\varinjlim}$''}}\alpha \otimes \beta\right)  \times  \left(\underset{I \times K}{\text{``${\varinjlim}$''}}\alpha \otimes \gamma\right) \\
	&\cong A \otimes_{\mathrm{I}} B \times A \otimes_{\mathrm{I}} C
	\end{align*}
	as desired.
\end{proof}

\subsection{The countable envelope of an exact category}
From now on, we endow $\mathcal{E}$ with the structure of an exact category (in the sense of Quillen). We are interested in the countable evelope $\mathrm{C}\mathcal{E}$, which is defined as a full subcategory of $\mathrm{Ind}(\mathcal{E})$. Let $I_0$ denote the category
\[\bullet \to \bullet \to \bullet \to \bullet \to \cdots\]
where we omit identities and compositions of morphisms.
\begin{dfn}[cf.\ \cite{kellerchain}*{Appendix B}] 
	The \emph{countable envelope $\mathrm{C}\mathcal{E}$ of $\mathcal{E}$} is defined as the full subcategory of $\mathrm{Ind}(\mathcal{E})$ consisting of all those objects isomorphic to one of the form $\text{``${\varinjlim}$''} \alpha$, where $\alpha: I_0 \to \mathcal{E}$ is a functor that maps all arrows of $I_0$ to inflations in $\mathcal{E}$.
\end{dfn}

\begin{rem}
	The embedding $\mathcal{E} \to \mathrm{Ind}(\mathcal{E})$ factors via $\mathrm{C}\mathcal{E}$ by choosing for an object $E \in \mathcal{E}$ the functor $\alpha_E$ that maps $I_0$ to the constant diagram
	\[E \xrightarrow{\mathrm{id}} E \xrightarrow{\mathrm{id}} E \xrightarrow{\mathrm{id}} E \to \cdots\]
	in $\mathcal{E}$.
	\label{remyonedfactorscountable}
\end{rem}

Keller shows in \cite{kellerchain}*{Appendix B} that $\mathrm{C}\mathcal{E}$ can be endowed with an exact structure as follows: 
\begin{thm}[\cite{kellerchain}*{Appendix B}]
	The following defines an exact structure on $\mathrm{C}\mathcal{E}$: a sequence of maps $A \to B \to C$ is a conflation if and only if it is isomorphic to a sequence 
	\[\text{``${\varinjlim}$''} \alpha \xrightarrow{\text{``${\varinjlim}$''} f}  \text{``${\varinjlim}$''} \beta \xrightarrow{\text{``${\varinjlim}$''} g} \text{``${\varinjlim}$''} \gamma \]
	where $\alpha,\beta,\gamma: I_0 \to \mathcal{E}$ are functors that send all maps of $I_0$ to inflations, and $f: \alpha \to \beta, g:\beta \to \gamma$ are morphisms of functors such that $\alpha(i) \xrightarrow{f(i)} \beta(i) \xrightarrow{g(i)} \gamma(i)$ is a conflation in $\mathcal{E}$ for all $i \in I_0$.
	\label{thmcountableexact}
\end{thm}

\begin{rem}
	It follows that the embedding $\mathcal{E} \to \mathrm{C}\mathcal{E}$ is exact.
\end{rem}

\begin{rem}
	In \cite{kellerchain}*{Appendix B}, the exact structure is actually defined on the category $\mathrm{Ind}'(\mathcal{E})$, but it defines an exact structure on the equivalent category $\mathrm{Ind}(\mathcal{E})$ as well.
\end{rem}

\subsection{Tensor exact categories}

\begin{dfn}
	A \emph{tensor exact category} is an exact category $\mathcal{E}$ equipped with a compatible symmetric monoidal structure $\otimes_{\mathcal{E}}$, i.e. the functors
	\[a \mapsto a \otimes_{\mathcal{E}} b\]
	are exact for all objects $b \in \mathcal{E}$.
	\label{dfntensorexactcat}
\end{dfn}

\begin{prop}
	For a tensor exact category $\mathcal{E}$, the countable envelope $\mathrm{C}\mathcal{E}$ naturally inherits the structure of a tensor exact category such that the embedding $\mathcal{E} \to \mathrm{C}\mathcal{E}$ is tensor exact.
	\label{propcetensex}
\end{prop}
\begin{proof}
	The symmetric monoidal structure on $\mathrm{C}\mathcal{E}$ is the restriction of the one on $\mathrm{Ind}(\mathcal{E})$ (see Proposition \ref{propindmonoidal}). For two functors $\alpha, \beta: I_0 \to \mathcal{E}$ with 
	\[\text{``${\varinjlim}$''} \alpha = A, \quad \text{``${\varinjlim}$''} \beta = B\] 
	we have by definition
	\[A \otimes_{\mathrm{I}} B = \underset{((U \to A),(V \to B)) \in \mathcal{E}_A \times \mathcal{E}_B}{\text{``${\varinjlim}$''}} U \otimes V \cong \underset{I_0 \times I_0}{\text{``${\varinjlim}$''}} \alpha \otimes \beta \cong \underset{I_0}{\text{``${\varinjlim}$''}} \alpha \otimes \beta\]
	where the first isomorphism follows from \cite{kaschap}*{Proposition 2.6.3 (ii)} and the second one follows as the diagonal functor $I_0\to I_0 \times I_0$ is cofinal (see \cite{kaschap}*{Corollary 3.2.3}). This proves that the tensor product of two objects in $\mathrm{C}\mathcal{E}$ is in $\mathrm{C}\mathcal{E}$ again. Indeed, the morphisms $\alpha(i) \otimes \beta(i) \to \alpha(j) \otimes \beta(j)$ are inflations for all objects $i,j \in I_0$ by the exactness property of $\otimes$.
	
	It remains to show that for $A \in \mathrm{C}\mathcal{E}$, the functor $A \otimes_{\mathrm{I}} -$ is exact. Let $\alpha,\beta,\gamma: I_0 \to \mathcal{E}$ be functors and $f: \alpha \to \beta, g: \beta \to \gamma$ be natural transformations such that 
	\[\alpha(i) \xrightarrow{f(i)} \beta(i) \xrightarrow{g(i)} \gamma(i)\] 
	is a conflation for all objects $i \in I_0$. If $A \cong \text{``${\varinjlim}$''} \delta$, then applying $A \otimes_{\mathrm{I}} -$ to the conflation
	\[\text{``${\varinjlim}$''} \alpha \xrightarrow{\text{``${\varinjlim}$''} f}  \text{``${\varinjlim}$''} \beta \xrightarrow{\text{``${\varinjlim}$''} g} \text{``${\varinjlim}$''} \gamma\]
	yields a sequence isomorphic to
	\[\text{``${\varinjlim}$''} \alpha \otimes \delta \xrightarrow{\text{``${\varinjlim}$''} f \otimes \id}  \text{``${\varinjlim}$''} \beta \otimes \delta \xrightarrow{\text{``${\varinjlim}$''} g \otimes \id} \text{``${\varinjlim}$''} \gamma \otimes \delta ~.\]
	As for all $i \in I_0$, the sequence
	\[\alpha(i) \otimes \delta(i) \xrightarrow{ f(i) \otimes \id} \beta(i) \otimes \delta(i) \xrightarrow{g(i) \otimes \id} \gamma(i) \otimes \delta(i)\]
	is a conflation by the exactness of the tensor product on $\mathcal{E}$, it follows that  $A \otimes_{\mathrm{I}} -$ is isomorphic to an exact functor and therefore exact itself.
\end{proof}

\begin{dfn}
	We say that a tensor exact category $\mathcal{E}$ satisfies \emph{the pushout product axiom} if for every two inflations $f: A \to B, g:C \to D$ in $\mathcal{E}$, the canonical morphism
	\[A \otimes D \coprod_{A \otimes C} B \otimes C \to B \otimes D\]
	\label{dfnpushoutprodaxiom}
	is an inflation.
\end{dfn}

Recall from \cite{buehlerexact}*{Example 13.11} that for any category $\mathcal{D}$ and an exact category $\mathcal{E}$, the category $\mathcal{E}^{\mathcal{D}}$ of functors $\mathcal{D} \to \mathcal{E}$ inherits an exact structure, where a sequence of natural transformations
\[F \to G \to H\]
is defined to be exact if $F(d) \to G(d) \to H(d)$ is exact in $\mathcal{E}$ for all objects $d \in \mathcal{D}$. We call this the \emph{pointwise exact structure on $\mathcal{E}^{\mathcal{D}}$}.
\begin{lma}
	Let $\mathcal{E}$ be a tensor exact category with tensor product $\otimes_{\mathcal{E}}$ and denote by $\tilde{\mathrm{C}}(\mathcal{E})$ the category of functors $\alpha: I_0 \to \mathcal{E}$, such that $\alpha$ maps all morphisms of $I_0$ to inflations, with the pointwise exact structure (see \cite{kellerchain}). Then $\tilde{\mathrm{C}}(\mathcal{E})$ with the pointwise tensor product $\otimes_{\tilde{\mathrm{C}}(\mathcal{E})}$ makes $\tilde{\mathrm{C}}(\mathcal{E})$ a tensor exact category. Furthermore, if $\mathcal{E}$ satisfies the pushout product axiom, then so does $\tilde{\mathrm{C}}(\mathcal{E})$. 
	\label{lmapushprodfunctcat}
\end{lma}
\begin{proof}
	It is clear that $\tilde{\mathrm{C}}(\mathcal{E})$ inherits a symmetric monoidal structure from $\mathcal{E}$: the associator, unitor and commutator isomorphisms are all given pointwise by the symmetric monoidal structure on $\mathcal{E}$ and they satisfy the required coherence conditions as they are satisfied for $\otimes_{\mathcal{E}}$. Note that the exactness properties of $\otimes_{\mathcal{E}}$ show that for two functors $\alpha,\beta \in \tilde{\mathrm{C}}(\mathcal{E})$, their tensor product $a \otimes_{\tilde{\mathrm{C}}(\mathcal{E})} b$ is again a functor that maps all morphisms of $I_0$ to inflations. The exactness properties of $\otimes_{\mathcal{E}}$ also imply that $\otimes_{\tilde{\mathrm{C}}(\mathcal{E})}$ has them as well and thus $\tilde{\mathrm{C}}(\mathcal{E})$ together with $\otimes_{\tilde{\mathrm{C}}(\mathcal{E})}$ is indeed a tensor exact category.
	
	Now let us assume that $\mathcal{E}$ satisfies the pushout product axiom. As we can compute pushouts in $\tilde{\mathrm{C}}(\mathcal{E})$ pointwise, it follows that the map in question from Definition \ref{dfnpushoutprodaxiom} is pointwise an inflation and therefore an inflation in $\tilde{\mathrm{C}}(\mathcal{E})$ by definition of the exact structure.
\end{proof}

\begin{prop}
	Assume $\mathcal{E}$ satisfies the pushout product axiom. Then the same holds true for $\mathrm{C}\mathcal{E}$.
	\label{propcepushprod}
\end{prop}
\begin{proof}
	Let us first remark that by Lemma \ref{lmapushprodfunctcat}, the functor category $\tilde{\mathrm{C}}(\mathcal{E})$ satisfies the pushout-product axiom. Furthermore, it is an immediate consequence of the definition of the exact structure on $\mathrm{C}\mathcal{E}$ and the tensor product $\otimes_{\mathrm{C}\mathcal{E}}$ that the functor 
	\[\text{``${\varinjlim}$''}:\tilde{\mathrm{C}}(\mathcal{E}) \to \mathrm{C}\mathcal{E}\] 
	is exact and preserves tensor products.
	
	Now, let $f:A \to A'$, $g:B \to B'$ be two inflations in $\mathrm{C}\mathcal{E}$. This means that there exist inflations $f': \alpha \to \alpha'$ and $g': \beta \to \beta'$ in $\tilde{\mathrm{C}}\mathcal{E}$ such that $f \cong \text{``${\varinjlim}$''}(f')$ and $g \cong \text{``${\varinjlim}$''}(g')$ (see Theorem \ref{thmcountableexact}). Look at the pushout diagram in $\tilde{\mathrm{C}}(\mathcal{E})$
	\[
	\xymatrix{
		\alpha \otimes_{\tilde{\mathrm{C}}\mathcal{E}} \beta \ar[r]^{\id \otimes g'} \ar[d]_{f' \otimes \id} & \alpha \otimes_{\tilde{\mathrm{C}}\mathcal{E}} \beta' \ar[d] \ar@/^1pc/[rdd]^{\id \otimes g'}& \\
		\alpha' \otimes_{\tilde{\mathrm{C}}\mathcal{E}} \beta \ar[r] \ar@/_/[rrd]_{f' \otimes \id} & \alpha' \otimes_{\tilde{\mathrm{C}}\mathcal{E}} \beta \coprod\limits_{\alpha \otimes_{\tilde{\mathrm{C}}\mathcal{E}} \beta} \alpha \otimes_{\tilde{\mathrm{C}}\mathcal{E}} \beta' \ar@{.>}[rd]^{h'} & \\
		&& \alpha' \otimes_{\tilde{\mathrm{C}}\mathcal{E}} \beta'
	}
	\]
	where $h'$ is an inflation as ${\tilde{\mathrm{C}}}\mathcal{E}$ satisfies the pushout product axiom. We now apply the functor $\text{``${\varinjlim}$''}$ to this diagram. As exact functors preserve pushouts along inflations (see \cite{buehlerexact}*{Proposition 5.2}) and $\text{``${\varinjlim}$''}$ commutes with the tensor products, we obtain a pushout diagram isomorphic to
	\[
	\xymatrix{
		A \otimes_{\mathrm{C}\mathcal{E}} B  \ar[r]^{\id \otimes g} \ar[d]_{f \otimes \id} & A \otimes_{\mathrm{C}\mathcal{E}} B' \ar[d] \ar@/^1pc/[rdd]^{\id \otimes g}& \\
		A' \otimes_{\mathrm{C}\mathcal{E}} B \ar[r] \ar@/_/[rrd]_{f \otimes \id} & A' \otimes_{\mathrm{C}\mathcal{E}} B \coprod\limits_{A \otimes_{\mathrm{C}\mathcal{E}} B} A \otimes_{\mathrm{C}\mathcal{E}} B' \ar@{.>}[rd]^{h} & \\
		&& A' \otimes_{\mathrm{C}\mathcal{E}} B'
	}
	\]
	where $h$ is an inflation as $\text{``${\varinjlim}$''}$ is exact. This finishes the proof.
\end{proof}

\subsection{Tensor Frobenius pairs}
Recall that a Frobenius category is an exact category with enough injective objects, such that the class of injective and projective objects coincide.

\begin{dfn}[see \cite{schlichtingnegative}*{Section 3.4}] 
	A \emph{Frobenius pair} $\mathbf{E}=(\mathcal{E},\mathcal{E}_0)$ is a strictly full, faithful and exact inclusion of Frobenius categories $\mathcal{E}_0 \hookrightarrow \mathcal{E}$ such that the projective-injective objects of $\mathcal{E}_0$ are mapped to the projective-injective objects of~$\mathcal{E}$.
	\label{dfnfrobpair}
\end{dfn}

We now give a symmetric monoidal version of Definition \ref{dfnfrobpair}.
\begin{dfn}
	A \emph{tensor Frobenius pair} $\mathbf{E}=(\mathcal{E},\mathcal{E}_0, \otimes)$ consists of a Frobenius pair $(\mathcal{E},\mathcal{E}_0)$ and a symmetric monoidal structure on $\mathcal{E}$ with tensor product $\otimes$, that makes $\mathcal{E}$ a tensor exact category and satisfies the following properties:
	\begin{enumerate}[label=(\roman*)]
		\item For all objects $A \in \mathcal{E}$, the functor $A \otimes -$ preserves the projective/injective objects of $\mathcal{E}$.
		\item $\mathcal{E}_0$ is a $\otimes$-ideal in $\mathcal{E}$, i.e.\ it is stable under tensoring with any object of $\mathcal{E}$. \label{axtensorideal}
		\item The tensor exact category $\mathcal{E}$ satisfies the pushout product axiom.
	\end{enumerate}
	\label{dfntensorfrobpair}
\end{dfn}

\begin{rem}
	In many examples, $\mathcal{E}$ will be a category of chain complexes over some exact category and $\mathcal{E}_0$ the subcategory of acyclic complexes. From this point of view, requiring that $\mathcal{E}_0$ is a $\otimes$-ideal says that $\otimes$ passes directly to the corresponding derived category.
	
	The pushout product axiom is there to make sure that $\otimes$ induces a product in the Waldhausen $\mathrm{K}$-theory of the Frobenius pair (see Lemma \ref{tensorwaldexact}).
\end{rem}

\begin{rem}
	Here is an example where axiom \ref{axtensorideal} of Definition \ref{dfntensorfrobpair} is \emph{not} satisfied: let $R\Mod$ be the abelian category of finitely generated modules over a commutative noetherian ring $R$ such that $\otimes_R$ is not an exact functor (i.e.\ $R$ is not absolutely flat). Consider $\mathrm{C^b}(R\Mod)$, the exact category of bounded chain complexes of finitely generated $R$-modules, with conflations the degree-wise split ones and $\mathrm{aC^b(R\Mod)}$, the exact subcategory of acyclic complexes. Then
	$(\mathrm{C^b}(R\Mod),\mathrm{aC^b(R\Mod)})$ is a Frobenius pair and the tensor product of chain complexes $\otimes_R$ makes this example almost a tensor Frobenius pair. However, $\mathrm{aC^b(R\Mod)}$ is not a tensor ideal as $\otimes_R$ is not exact.
\end{rem}

If $\mathcal{E}$ is a Frobenius category, $\mathrm{C}\mathcal{E}$ is one as well, with the exact structure from Theorem \ref{thmcountableexact}, according to \cite{schlichtingnegative}*{Section 4}. It follows that for a Frobenius pair $\mathbf{E}=(\mathcal{E},\mathcal{E}_0)$, its countable envelope 
$\mathrm{C}\mathbf{E}:=(\mathrm{C}\mathcal{E},\mathrm{C}\mathcal{E}_0)$ is again a Frobenius pair. We want to prove an analogous statement for tensor Frobenius pairs.

\begin{thm}
	Let $\mathbf{E} = (\mathcal{E},\mathcal{E}_0,\otimes)$ be a tensor Frobenius pair. Then its \emph{countable envelope} $\mathrm{C}\mathbf{E} := (\mathrm{C}\mathcal{E},\mathrm{C}\mathcal{E}_0,\otimes_{\mathrm{I}})$ is a tensor Frobenius pair.
	\label{thmceistfrobpair}
\end{thm}
\begin{proof}
	We know that $(\mathrm{C}\mathcal{E},\mathrm{C}\mathcal{E}_0)$ is a Frobenius pair and Proposition \ref{propcetensex} gives a symmetric monoidal structure on $\mathrm{C}\mathcal{E}$ with tensor product $\otimes_{\mathrm{I}}$ that makes $\mathrm{C}\mathcal{E}$ a tensor exact category. Furthermore, $\mathrm{C}\mathcal{E}$ will satisfy the pushout product axiom by Proposition \ref{propcepushprod}. 
	
	In order to show that $\mathrm{C}\mathcal{E}_0$ is a $\otimes_{\mathrm{I}}$-ideal in $\mathrm{C}\mathcal{E}$, let $A \cong \text{``${\varinjlim}$''} \alpha, B \cong \text{``${\varinjlim}$''} \beta$ for two functors $\alpha: I_0 \to \mathcal{E}, \beta: I_0 \to \mathcal{E}_0$. Then
	\[A \otimes_{\mathrm{I}} B \cong \text{``${\varinjlim}$''} \alpha \otimes \beta \]
	and as $\mathcal{E}_0$ is a $\otimes$-ideal in $\mathcal{E}$, it follows that $\alpha \otimes \beta$ has image $\mathcal{E}_0$ and thus $A \otimes_{\mathrm{I}} B \in \mathrm{C}\mathcal{E}_0$.
	
	It remains to prove that $A \otimes_{\mathrm{I}} -$ preserves the projective-injective objects of $\mathrm{C}\mathcal{E}$ which are given as direct summands of objects isomorphic to $\text{``${\varinjlim}$''} \iota$ where $\iota: I_0 \to \mathcal{E}\mathrm{-prinj}$ takes values in the full subcategory of projective-injective objects of $\mathcal{E}$ (see \cite{schlichtingnegative}*{Definition 4.3}). For such $\iota$ and any $\text{``${\varinjlim}$''} \alpha \in \mathrm{C}\mathcal{E}$ we have
	\[\left(\text{``${\varinjlim}$''} \alpha \right) \otimes_{\mathrm{I}} \left(\text{``${\varinjlim}$''} \iota\right) \cong \left(\text{``${\varinjlim}$''} \alpha \otimes \iota \right)\]
	and as $\mathcal{E}$ is a tensor Frobenius pair we see that the functor $\alpha \otimes \iota$ takes values in $\mathcal{E}\mathrm{-prinj}$. Thus for any $A \in \mathrm{C}\mathcal{E}$, $A \otimes_{I} - $ preserves objects isomorphic to $\text{``${\varinjlim}$''} \alpha$ where $\alpha: I_0 \to \mathcal{E}\mathrm{-prinj}$. As it is an additive functor it also preserves their direct summands. We conclude that $A \otimes_{\mathrm{I}} -$ preserves the projective-injective objects of $\mathrm{C}\mathcal{E}$ which finishes the proof.
\end{proof}

\end{appendices}

\small{\textsc{Sebastian Klein, Departement Wiskunde-Informatica, Universiteit Antwerpen, Middelheimcampus, Middelheimlaan 1, 2020 Antwerp, Belgium}\\
\textit{E-mail address:} \texttt{sebastian.klein@uantwerpen.be}}


\begin{bibdiv}
	\addcontentsline{toc}{section}{References}
	\begin{biblist}
		
		\bib{SGA4-1}{book}{
			author={Artin, Michael},
			author={Grothendieck, Alexander},
			author={Verdier, Jean-Louis},
			title={Th\'eorie de topos et cohomologie \'etale des schemas {I}},
			series={Lecture Notes in Mathematics},
			publisher={Springer},
			date={1971},
			volume={269},
		}
		
		\bib{balmer2005spectrum}{article}{
			author={Balmer, Paul},
			title={The spectrum of prime ideals in tensor triangulated categories},
			date={2005},
			ISSN={0075-4102},
			journal={J. Reine Angew. Math.},
			volume={588},
			pages={149\ndash 168},
			url={http://dx.doi.org/10.1515/crll.2005.2005.588.149},
		}
		
		\bib{balmerfiltrations}{article}{
			author={Balmer, Paul},
			title={Supports and filtrations in algebraic geometry and modular
				representation theory},
			date={2007},
			ISSN={0002-9327},
			journal={Amer. J. Math.},
			volume={129},
			number={5},
			pages={1227\ndash 1250},
			url={http://dx.doi.org/10.1353/ajm.2007.0030},
		}
		
		\bib{balmer2010tensor}{inproceedings}{
			author={Balmer, Paul},
			title={Tensor triangular geometry},
			date={2010},
			booktitle={Proceedings of the {I}nternational {C}ongress of
				{M}athematicians. {V}olume {II}},
			publisher={Hindustan Book Agency},
			address={New Delhi},
			pages={85\ndash 112},
		}
		
		\bib{balmerchow}{article}{
			author={Balmer, Paul},
			title={Tensor triangular {C}how groups},
			date={2013},
			ISSN={0393-0440},
			journal={J. Geom. Phys.},
			volume={72},
			pages={3\ndash 6},
			url={http://dx.doi.org/10.1016/j.geomphys.2013.03.017},
		}
		
		\bib{balschlichidem}{article}{
			author={Balmer, Paul},
			author={Schlichting, Marco},
			title={Idempotent completion of triangulated categories},
			date={2001},
			ISSN={0021-8693},
			journal={J. Algebra},
			volume={236},
			number={2},
			pages={819\ndash 834},
			url={http://dx.doi.org/10.1006/jabr.2000.8529},
		}
		
		\bib{bencarlrick}{article}{
			author={Benson, David J.},
			author={Carlson, Jon F.},
			author={Rickard, Jeremy},
			title={Thick subcategories of the stable module category},
			journal={Fund. Math.},
			volume={153},
			date={1997},
			number={1},
			pages={59--80},
			issn={0016-2736},
		}
		
		\bib{sga6}{book}{
			editor={Berthelot, Pierre},
			editor={Grothendieck, Alexander},
			editor={Illusie., Luc},
			title={Th\'eorie des intersections et th\'eor\`eme de {R}iemann-{R}och},
			series={Lecture Notes in Mathematics, Vol. 225},
			publisher={Springer-Verlag},
			address={Berlin},
			date={1971},
			note={S{\'e}minaire de G{\'e}om{\'e}trie Alg{\'e}brique du Bois-Marie
				1966--1967 (SGA 6)},
		}
		
		\bib{bredonsheaf}{book}{
			author={Bredon, Glen~E.},
			title={Sheaf theory},
			publisher={McGraw-Hill Book Co.},
			address={New York},
			date={1967},
		}
		
		\bib{bkssupport}{article}{
			author={Buan, Aslak~Bakke},
			author={Krause, Henning},
			author={Solberg, {\O}yvind},
			title={Support varieties: an ideal approach},
			date={2007},
			ISSN={1532-0073},
			journal={Homology, Homotopy Appl.},
			volume={9},
			number={1},
			pages={45\ndash 74},
			url={http://projecteuclid.org/getRecord?id=euclid.hha/1175791087},
		}
		
		\bib{buehlerexact}{article}{
			author={B{\"u}hler, Theo},
			title={Exact categories},
			date={2010},
			ISSN={0723-0869},
			journal={Expo. Math.},
			volume={28},
			number={1},
			pages={1\ndash 69},
			url={http://dx.doi.org/10.1016/j.exmath.2009.04.004},
		}
		
		\bib{dayclosedcats}{article}{
			author={Day, Brian},
			title={On closed categories of functors},
			conference={
				title={Reports of the Midwest Category Seminar, IV},
			},
			book={
				series={Lecture Notes in Mathematics, Vol. 137},
				publisher={Springer, Berlin},
			},
			date={1970},
			pages={1--38},
		}
		
		\bib{delignecattak}{article}{
			author={Deligne, Pierre},
			title={Cat\'egories tannakiennes},
			conference={
				title={The Grothendieck Festschrift, Vol.\ II},
			},
			book={
				series={Progr. Math.},
				volume={87},
				publisher={Birkh\"auser Boston, Boston, MA},
			},
			date={1990},
			pages={111--195},
		}
		
		\bib{graysonproduct}{article}{
			author={Grayson, Daniel~R.},
			title={Products in {K}-theory and intersecting algebraic cycles},
			date={1978},
			ISSN={0020-9910},
			journal={Invent. Math.},
			volume={47},
			number={1},
			pages={71\ndash 83},
		}
		
		\bib{haiemb}{article}{
			author={H{\'a}i, Ph{\`u}ng~H{\^{o}}},
			title={An embedding theorem for abelian monoidal categories},
			date={2002},
			ISSN={0010-437X},
			journal={Compositio Math.},
			volume={132},
			number={1},
			pages={27\ndash 48},
			url={http://dx.doi.org/10.1023/A:1016076714394},
		}
		
		\bib{kaschap}{book}{
			author={Kashiwara, Masaki},
			author={Schapira, Pierre},
			title={Categories and sheaves},
			series={Grundlehren der Mathematischen Wissenschaften},
			publisher={Springer-Verlag},
			address={Berlin},
			date={2006},
			volume={332},
			ISBN={978-3-540-27949-5; 3-540-27949-0},
		}
		
		\bib{kellerchain}{article}{
			author={Keller, Bernhard},
			title={Chain complexes and stable categories},
			date={1990},
			ISSN={0025-2611},
			journal={Manuscripta Math.},
			volume={67},
			number={4},
			pages={379\ndash 417},
			url={http://dx.doi.org/10.1007/BF02568439},
		}
		
		\bib{kellerderuse}{article}{
			author={Keller, Bernhard},
			title={Derived categories and their uses},
			conference={
				title={Handbook of algebra, Vol.\ 1},
			},
			book={
				publisher={North-Holland, Amsterdam},
			},
			date={1996},
			pages={671--701},
		}
		
		\bib{kleinchow}{misc}{
			author={Klein, Sebastian},
			title={Chow groups for tensor triangulated categories},
			date={2014},
			note={preprint, arXiv:1301.0707v2 [math.AG]},
		}
		
		\bib{maclanecwm}{book}{
			author={Mac~Lane, Saunders},
			title={Categories for the working mathematician},
			edition={Second},
			series={Graduate Texts in Mathematics},
			publisher={Springer-Verlag, New York},
			date={1998},
			volume={5},
			ISBN={0-387-98403-8},
		}
		
		\bib{panin}{article}{
			author={Panin, Ivan A.},
			title={The equicharacteristic case of the {G}ersten conjecture},
			date={2003},
			ISSN={0371-9685},
			journal={Tr. Mat. Inst. Steklova},
			volume={241},
			number={Teor. Chisel, Algebra i Algebr. Geom.},
			pages={169\ndash 178},
		}
		
		\bib{quillenhigher}{article}{
			author={Quillen, Daniel},
			title={Higher algebraic $K$-theory. I},
			conference={
				title={Algebraic K-theory, I: Higher K-theories},
				address={Proc. Conf., Battelle Memorial Inst., Seattle, Wash.},
				date={1972},
			},
			book={
				publisher={Springer, Berlin},
			},
			date={1973},
			pages={85--147. Lecture Notes in Math., Vol. 341},
		}
		
		\bib{schlichtingknote}{article}{
			author={Schlichting, Marco},
			title={A note on {K}-theory and triangulated categories},
			date={2002},
			ISSN={0020-9910},
			journal={Invent. Math.},
			volume={150},
			number={1},
			pages={111\ndash 116},
			url={http://dx.doi.org/10.1007/s00222-002-0231-1},
		}
		
		\bib{schlichtingnegative}{article}{
			author={Schlichting, Marco},
			title={Negative {K}-theory of derived categories},
			date={2006},
			ISSN={0025-5874},
			journal={Math. Z.},
			volume={253},
			number={1},
			pages={97\ndash 134},
			url={http://dx.doi.org/10.1007/s00209-005-0889-3},
		}

		\bib{thomasonclassification}{article}{
			author={Thomason, Robert~W.},
			title={The classification of triangulated subcategories},
			date={1997},
			journal={Compositio Mathematica},
			volume={105},
			number={1},
			pages={1\ndash 27},
		}
		
		\bib{thomason-trobaugh}{article}{
			author={Thomason, Robert W.},
			author={Trobaugh, Thomas},
			title={Higher algebraic K-theory of schemes and of derived categories},
			conference={
				title={The Grothendieck Festschrift, Vol.\ III},
			},
			book={
				series={Progr. Math.},
				volume={88},
				publisher={Birkh\"auser Boston, Boston, MA},
			},
			date={1990},
			pages={247--435},
		}
		
		\bib{waldgenprods}{article}{
			author={Waldhausen, Friedhelm},
			title={Algebraic K-theory of generalized free products. I},
			journal={Ann. of Math. (2)},
			volume={108},
			date={1978},
			number={1},
			pages={135--204},
			issn={0003-486X},
		}
		
		\bib{waldktheory}{article}{
			author={Waldhausen, Friedhelm},
			title={Algebraic K-theory of spaces},
			conference={
				title={Algebraic and geometric topology},
				address={New Brunswick, N.J.},
				date={1983},
			},
			book={
				series={Lecture Notes in Math.},
				volume={1126},
				publisher={Springer, Berlin},
			},
			date={1985},
			pages={318--419},
		}

	\end{biblist}
\end{bibdiv}
\end{document}